\documentclass[a4paper,11pt]{article}
\usepackage[utf8]{inputenc}

\usepackage[english]{babel}
\usepackage{amsmath,amssymb, amsfonts,amsthm}
\usepackage[dvips]{graphicx}
\usepackage{color}
\usepackage{colordvi}
\usepackage{dsfont}
\usepackage{enumitem}
\usepackage{needspace}
\usepackage{pdflscape}
\usepackage{subfigure}
\usepackage{mathtools}

\DeclareMathOperator*{\argmin}{arg\,min}

\newcommand{\N}{\mathbb{N}}

\newcommand{\R}{\mathbb{R}}
\newcommand{\PP}{\mathbb{P}}

\newcommand{\1}{1\! \mathrm{l}}
\newcommand{\E}{\mathbb{E}}

\newtheorem{theo}{Theorem}

\newtheorem{prop}{Proposition}
\newtheorem{cor}{Corollary}
\newtheorem*{assumption}{Assumption}
\newtheorem{lemma}{Lemma}

\usepackage{palatino}
\usepackage{colortbl}
\usepackage{tikz}
\usepackage{tkz-euclide}
\usetikzlibrary{shapes,arrows}

\usepackage{authblk}

\author[1]{Elisabeth Gassiat}
\author[2]{Judith Rousseau}
\author[3]{Elodie Vernet}
\affil[1]{Laboratoire de Math\'ematiques d'Orsay, Univ. Paris-Sud, CNRS, Universit\'{e} Paris-Saclay, 91405 Orsay, France}
\affil[2]{CEREMADE, Universit\'{e} Paris Dauphine, Paris, Franc}
\affil[3]{CMAP, \'{E}cole Polytechnique, route de Saclay, 91128 Palaiseau Cedex, France}
\date{}                     
\title{Efficient semiparametric estimation and model selection for multidimensional mixtures}

\begin{document}

\maketitle

\begin{abstract}
In this paper, we consider nonparametric multidimensional finite mixture models and we are interested in
the semiparametric estimation of the population weights. Here, the i.i.d. observations are assumed to have at least three components which are independent given the population. 
We approximate the semiparametric model 
by projecting the conditional distributions on step functions associated to some partition. 
Our first main result is that if we refine the partition slowly enough, the associated sequence of maximum likelihood estimators of the weights is asymptotically efficient, and the posterior distribution of the weights, when using a Bayesian procedure, satisfies  
a semiparametric Bernstein-von Mises theorem. We then propose a cross-validation like method to select the partition in a finite horizon.  Our second main result is that the proposed procedure satisfies an oracle inequality. Numerical experiments  on simulated data illustrate our theoretical results.
\end{abstract}


\tableofcontents

\section{Introduction}

We consider in this paper multidimensional mixture models that describe the  probability distribution  of a random vector $X$ with at least three coordinates. The model is a probability mixture of $k$ populations such that the coordinates of $X$ can be grouped into 3 blocks of  random variables which are conditionally independent given the population.
We call emission distributions the conditional distributions of the coordinates and $\theta$ the parameter that contains the probability weights of each population. 
It has been known for some time that such a model is identifiable under weak assumptions. 
When the coordinates of $X$ take finitely many values, Kruskal \cite{MR0444690} in 1977 provided an algebraic sufficient condition under which he proved identifiability. See also \cite{rhodes2010concise}.
Kruskal's result was recently used by \cite{allman2009identifiability} to obtain identifiability under almost no assumption on the possible emission distributions: only the fact that, for each coordinate, the $k$ emission distributions are 
linearly independent. Spectral methods were proposed by \cite{kakade14}, which allowed \cite{MR3476609} to derive estimators of the emission densities having the minimax rate of convergence when the smoothness of the emission densities is known. Moreover, \cite{BoJoRo16} proposes an estimation procedure  in the case of repeated measurements (where the emission distributions of each coordinate given a population are the same). 
\\

Our paper focuses on the semiparametric estimation of the population weights 
when nothing is known about the emission distributions. This is a semiparametric model, where the finite dimensional parameter of interest is $\theta$ and the infinite dimensional nuisance parameters are the emission distributions. 
In applications, the populations weights have direct interpretation. As an example, in \cite{MR2915165} the problem is to estimate the proportion   of cells of different types for diagnostic purposes, on the basis of flow cytometry data. Those data give the intensity of several markers  responses and may be modelled as multidimensional mixtures.

We are in particular interested in constructing optimal procedures for the estimation of $\theta$.  Optimal may be understood as efficient, in Le Cam's theory point of view which is about asymptotic distribution and asymptotic (quadratic) loss. See \cite {lecam_yang00}, \cite{MR1245941}, \cite{AadBook}, \cite{AadStFlour}.
The first question is: is the parametric rate attainable in the semiparametric setting? We know here, for instance using spectral estimates, that the parametric rate is indeed attainable. Then, the loss due to the nuisance parameter may be seen in the efficient Fisher information and efficient estimators are asymptotically equivalent to the empirical process on efficient influence functions. The next question is thus: how can we construct asymptotically efficient estimators? In the parametric setting, maximum likelihood estimators (MLEs) do the job, but the semiparametric situation is more difficult, because one has to deal with the unknown nuisance parameter, see the theorems in chapter 24 of  \cite{AadBook} where it is necessary to control various bias/approximation terms.

From a Bayesian perspective,  the issue is  the validity of the Bernstein-Von Mises property of the marginal  posterior distribution of the parameter of interet $\theta$. In other words: is the marginal posterior distribution of $\theta$ asymptotically Gaussian? Is it asymptotically centered around an efficient estimator? Is the asymptotic variance of the posterior distribution the inverse of the efficient Fisher information matrix? Semiparametric Bernstein-Von Mises theorems have been the subject of recent research, see  \cite{shen:2002}, \cite{boucherongassiat09}, \cite{JuRi}, \cite{icbvm}, \cite{icsan}, \cite{bickel:kleijn}, \cite{deblasi09} and \cite{JuRi}.
\\

The results of our paper are twofold: first we obtain asymptotically efficient semiparametric estimators using a likelihood strategy, then we propose a data driven method to perform the strategy in a finite horizon with an oracle inequality as theoretical guarantee.
\\

Let us describe our ideas.\\

For the multidimensional mixture model we consider, we will take advantage of the fact that, we can construct a parametric mixture model based on an approximation of the emission densities by piecewise constant functions  - i.e histograms - which acts as a correct model for a coarsened version of the observations (the observations are replaced by the number of points in each grid of the histograms).  So that as far as the parameter of interest is concerned, namely the weights of the mixture, this approximate model is in fact well specified, in particular the Kullback-Leibler divergence between the true distribution and the approximate model is minimized at the true value of the parameter of interest, see Section \ref{subsec:notations} for more details. 
For each of these finite dimensional models, the parameter of interest, i.e. the weights of the mixture,  may then be efficiently estimated within the finite dimensional model.  Then, under  weak assumptions, and using the fact that one can approximate any density on $[0,1]$ by such histograms based on partitions with radius (i.e. the size of the largest bin) going to zero, it is possible to prove that asymptotically efficient semiparametric estimators may be built using the sequence of MLEs in a growing (with sample size) sequence of approximation models. In the same way, using Bayesian posteriors in the growing sequence of approximation models, one gets  a Bernstein-Von Mises result. One of the important implications of the Bernstein-von Mises property is that credible regions, such as highest posterior density regions or credible ellipsoids are also confidence regions. In the particular case of the semiparametric mixtures, this is of great interest, since the construction of a confidence region is not  necessarily trivial.
This is our first main result which is stated in Theorem \ref{th:FreeLunch}:
 by considering partitions  refined slowly enough when the number of observations increases, we can derive efficient estimation procedures for the parameter of interest $\theta$ and in the Bayesian approach for a marginal posterior distribution on $\theta$ which satisfies the renowned Bernstein-von Mises property.
 
 We still need however in  practice  to choose a good partition, for a finite sample size. 
This can be viewed as a model selection problem. There is now a huge literature on model selection, both in the frequentist and in the Bayesian literature. Roughly speaking the methods can be split into two categories: penalized likelihood types of approaches, which include in particular AIC (Akaike's Information Criterion), BIC (Bayesian Information Criterion), MDL (Minimum Description Length) and marginal likelihood (Bayesian) criteria or approaches which consist in estimating the risk of the estimator in each model using  for instance bootstrap or cross-validation methods. In all these cases theory and practice are nowadays well grounded, see for instance \cite{hansen:yu:01}, \cite{robert01}, \cite{barbe:bertail:95},  \cite{massart03}, \cite{maugis2008slope}, \cite{MR2602303}, \cite{ClHj08}, \cite{To10}. Most of the existing results above cover parametric or nonparametric models. Penalized likelihoods in particular target models which are best in terms of Kullback-Leibler divergences typically and therefore aim at estimating the whole nonparametric parameter. Risk estimation via bootstrap or cross-validation methods are more naturally defined  in semiparametric (or more generally set-ups with  nuisance parameters) models, however the theory remains quite limited in cases where the estimation strategy is strongly non linear as encountered here.

The idea is to estimate the risk of the estimator in each approximation model, and then select the model with the smallest estimated risk.  We propose to use a cross-validation method similar to the one proposed in \cite{VdLaan}. To get theoretical results on such a strategy,  the usual basic tool is to write the cross-validation criterion as a function of the empirical distribution which is not possible in our semiparametric setting. We thus divide the sample in non overlapping blocks of size $a_n$ ($n$ being the  the sample size) to define the cross validation criterion. This enables us to prove our second main result: Theorem \ref{th:construction_Mn} which states an oracle inequality on  the quadratic risk associated with a sample of  size $a_n$ observations,  and which also leads to a criterion to select $a_n$.  Simulations indicate moreover that the approach behaves well in practice.\\ 

In Section \ref{sec:eff}, we first describe the model, set the notations and our basic assumptions.
We recall the semiparametric tools in
Section \ref{subsec:efficiency}, where we define the score functions and the efficient Fisher information matrices.
Using the fact that spectral estimators are smooth functions of the empirical distribution of the observations, we obtain that, for large enough approximation model, the efficient Fisher information matrix is full rank, see
Proposition \ref{prop:fisher}. Intuition says that with better approximation spaces, more is known about all parameters of the distribution, in particular about $\theta$. We prove in
Proposition \ref{prop:FisherGrow} that indeed,  when the partition is refined, the Fisher information associated to this partition increases. 
This leads to  our main general result presented in Theorem 1,
Section \ref{subsec:gene}: it is possible to let the approximation parametric models grow with the sample size so that the sequence of maximum likelihood estimators are asymptotically efficient in the semiparametric model and so that a semiparametric Bernstein-von Mises Theorem holds. To prove this result we prove in particular,  in
Lemma \ref{lem:convscore}, that semiparametric score functions and semiparametric efficient Fisher information matrix are the limits of the parametric ones obtained in the approximation parametric models, which has interest in itself, given the non explicit nature of semi-parametric efficient score functions and information matrices in such models. This implies in particular  that the semiparametric efficient Fisher information matrix is full rank follows.
In
Section \ref{sec:modelselection}, we propose a model selection approach to select the number of bins.  We first discuss in Section \ref{se:reasons_model_selection} the reasons to perform model selection and the fact that choosing a too large approximation space does not work, see Proposition \ref{prop:limit_thetaM}
and
Corollary \ref{co:bound_for M_n}. Then we propose in 
Section \ref{se:criterion_model_selection} our cross-validation criterion, for which we prove an oracle inequality in
Theorem \ref{th:construction_Mn}
and
Proposition \ref{prop:alm_oracle_ineq}.
Results of simulations are described in
Section \ref{sec:simu}, where we investigate several choices of the number and length of blocks for performing cross validation, and investigate practically also V-fold strategies. In Section \ref{sec:discussion} we present  possible extensions, open questions and further work.  Finally 
Section \ref{sec:proofs} is dedicated to proofs of intermediate propositions and lemmas.

\section{Asymptotic efficiency}
\label{sec:eff}

\subsection{Model and notations}
\label{subsec:notations}

Let $(X_{n})_{n\geq 1}$ be a sequence of independent and identically distributed random variables taking values in  the product of at least three compact subsets of Euclidean spaces which, for the sake of simplicity, we will set as $[0,1]^{3}$. We assume that the possible marginal distribution of an observation $X_{n}$, $n\geq 1$, is a population mixture of $k$ distributions such that, given the population, the coordinates are independent and have some density with respect to the Lebesgue measure  on $[0,1]$.
The possible densities of $X_{n}$, $n\geq 1$, are, if $\mathbf{x}=(x_{1},x_{2},x_{3})\in [0,1]^{3}$:
\begin{equation}
\label{eq:model1}
g_{\theta,\mathbf{f}}(\mathbf{x})=\sum_{j=1}^{k}\theta_{j}\prod_{c=1}^{3}f_{j,c}(x_{c}), \quad \sum_{j=1}^{k}\theta_{j}=1, \quad \theta_j \geq 0, \, \forall j 
\end{equation}
Here, $k$ is the number of populations, $\theta_j$ is the  probability to belong to population $j$ for $j \leq k$ and we set $\theta=(\theta_{1},\dots,\theta_{k-1})$. For each $j=1,\ldots,k$, $f_{j,c}$, $c=1,2,3$, is the density of the $c$-th coordinate of the observation, given that the observation comes from population $j$, and we set 
$\mathbf{f}=((f_{j,c})_{1\leq c \leq 3})_{1\leq j \leq k}$. We denote by $\PP^{\star}$ the true (unknown) distribution of the sequence $(X_{n})_{n\geq 1}$, such that $\PP^{\star}=P_{\theta^{\star},\mathbf{f}^{\star}}^{\otimes \N}$, $dP_{\theta^{\star},\mathbf{f}^{\star}}(\mathbf{x})=g_{\theta^{\star},\mathbf{f}^{\star}}(\mathbf{x})d\mathbf{x}$, for some 
$\theta^{\star}\in\Theta$ and $\mathbf{f}^{\star}\in{\cal F}^{3k}$, where $\Theta$ is  the set of possible parameters $\theta$ and ${\cal F}$ the set of probability densities on $[0,1]$.\\

\noindent
We  approximate the densities by step functions on some partitions  of $[0,1]$. We assume that we have a collection of partitions ${\cal I}_{M}$, $M\in {\cal M}$,  ${\cal M}\subset \N$, so that
for each $M\in {\cal M}$, 
${\cal I}_{M}=(I_{m})_{1\leq m \leq M}$ is a partition of $[0,1]$ by Borel sets. 
Then $I_{m}$ changes when $M$ changes.
For each $M\in {\cal M}$,  we now consider the model of possible densities
\begin{equation}
\label{eq:model2}
g_{\theta,\boldsymbol{\omega};M}(\mathbf{x})=\sum_{j=1}^{k}\theta_{j}\prod_{c=1}^{3}\left(\sum_{m=1}^{M}\frac{\omega_{j,c,m}}{|I_{m}|}\1_{I_{m}} (x_{c})\right).
\end{equation}
Here, $|I|$ is the Lebesgue measure of the set $I$, $\boldsymbol{\omega}=(\omega_{j,c,m})_{1\leq m \leq M-1,1\leq c\leq 3,1\leq j \leq k}$, and for each $j=1,\ldots,k$, each $c=1,2,3$, each $m=1,\ldots,M-1$, 
$\omega_{j,c,m}\geq 0$, $\sum_{m=1}^{M-1}\omega_{j,c,m} \leq 1$, and we denote $\omega_{j,c,M}=1-\sum_{m=1}^{M-1}\omega_{j,c,m}$. 

\noindent
We denote $\Omega_{M}$ the set of possible parameters $\boldsymbol{\omega}$ when using  model (\ref{eq:model2}) with the partition  ${\cal I}_{M}$.

For any partition ${\cal I}_{M}$, any
$\omega=(\omega_{m})_{1\leq m \leq M-1}$ such that $\omega_{m}\geq 0$, $m=1,\ldots,M$, with $\omega_{m}=1-\sum_{m=1}^{M-1}\omega_{m}$, denote $f_{\omega}$ the step function given by
\begin{equation}
\label{fomega}
f_{\omega}(x)=\sum_{m=1}^{M}\frac{\omega_{m}}{|I_{m}|}\1_{I_{m}} (x), \quad \omega_{m}=\int_{0}^{1}f_{\omega}(u)\1_{I_{m}} (u)du.
\end{equation}
When $\boldsymbol{\omega} \in \Omega_M$, let $\mathbf{f}_{\boldsymbol{\omega}}=((f_{\omega_{j,c}})_{1\leq c \leq 3})_{1\leq j \leq k}$.\\

 \noindent
Possible extensions of our results to model \eqref{eq:model1} with non compact support, or with more than three coordinates, or with multivariate coordinates, and to model \eqref{eq:model2} with different sequences of partitions for each coordinate are discussed in Section \ref{sec:discussion}.\\

An interesting feature of  step function approximation is that the Kullback-Leibler divergence $KL((\theta^\star,\mathbf{f}^\star),(\theta,\mathbf{f}_{\boldsymbol{\omega}}))$ between the distribution with density $g_{(\theta^\star,\mathbf{f}^\star)}$ and that with density $g_{(\theta,{\boldsymbol{\omega}};M)}$, when $(\theta,{\boldsymbol{\omega}})\in \Theta\times\Omega_M$, is minimised at 
$$
\theta=\theta^{\star},\;\boldsymbol{\omega}=\boldsymbol{\omega}^{\star}_{M}:=\left( \int_{0}^{1}f_{j,c}^{\star}(u) \1_{I_{m}} (u)du\right)_{1\leq m \leq M-1,1\leq c\leq 3,1\leq j \leq k}.
$$
Indeed
\begin{equation*}
\begin{split}
KL&((\theta^\star,\mathbf{f}^\star),(\theta,\mathbf{f}_{\boldsymbol{\omega}}))
 := \E^\star \left[ \log\left(\frac{g_{\theta^{\star},\mathbf{f}^{\star}}(X)}{g_{\theta,\boldsymbol{\omega};M}(X) }\right) \right]  \\
 &
 = \E^\star \left[ \log\left(\frac{g_{\theta^{\star},\mathbf{f}^{\star}}(X)}{g_{\theta^{\star},\boldsymbol{\omega}^{\star};M}(X) }\right) \right]
 +
\E^\star \left[ \log\left(\frac{g_{\theta^{\star},\boldsymbol{\omega}^{\star};M}(X)}{g_{\theta,\boldsymbol{\omega};M}(X) }\right) \right]
 \\ 
 &
 = KL((\theta^\star,\mathbf{f}^\star),(\theta^{\star},\mathbf{f}_{\omega^{\star}}))
 +
  \E^\star \bigg( \sum_{\substack{1\leq m_1 \leq M\\ 1\leq m_2 \leq M\\1\leq m_3 \leq M}}  \prod_{c=1}^3 \1_{X_c\in I_{m_c}}  \log\left(\frac{g_{\theta^{\star},\boldsymbol{\omega}^{\star};M}(X)}{g_{\theta,\boldsymbol{\omega};M}(X) }\right) \bigg)
 \\
 &=
 KL((\theta^\star,\mathbf{f}^\star),(\theta^\star,\mathbf{f}_{{\boldsymbol{\omega}}^\star}))
 +
 KL((\theta^\star,\mathbf{f}_{{\boldsymbol{\omega}}^\star}),(\theta,\mathbf{f}_\omega)).
\end{split}
\end{equation*}
This particularity can also be obtained considering $Y_{i}:=((\1_{I_{m}}(X_{i,c}))_{1\leq m \leq M})_{1\leq c\leq 3}$  when the $X_i$'s have probability density \eqref{eq:model1}.
Indeed,   the density of the observations $Y_i$ is exactly the probability density \eqref{eq:model2} with
$$
\theta=\theta^{\star},\;\boldsymbol{\omega}=\boldsymbol{\omega}^{\star}_{M}
$$
This cornerstone property is specific to the chosen approximation, i.e. the step function approximation.\\
\noindent
Let $\ell_{n}(\theta,\boldsymbol{\omega};M)$ be the log-likelihood using model (\ref{eq:model2}), that is
$$
\ell_{n}(\theta,\boldsymbol{\omega};M)=\sum_{i=1}^{n}\log g_{\theta,\boldsymbol{\omega};M}(X_{i}).
$$ 
We  denote, for each $M\in{\cal M}$,
$(\widehat{\theta}_{M}, \widehat{{\boldsymbol{\omega}}}_{M})$ the maximum likelihood estimator (shortened as MLE), that is a maximizer of $\ell_{n}(\theta,\boldsymbol{\omega};M)$ over $\Theta\times \Omega_{M}$.\\

\noindent
Let  $\Pi_{M}$ denote a prior distribution on the parameter space $\Theta\times\Omega_{M}$. The posterior distribution
$\Pi_{M}(\cdot\vert X_{1},\ldots,X_{n})$ is defined as follows. For any Borel subset $A$ of $\Theta\times\Omega_{M}$,
$$
\Pi_{M}(A\vert X_{1},\ldots,X_{n})=\frac{\int_{A}  \prod_{i=1}^n g_{\theta,\boldsymbol{\omega};M}(X_{i})d\Pi_{M}(\theta,\boldsymbol{\omega})}{\int_{\Theta\times\Omega_{M}}  \prod_{i=1}^n g_{\theta,\boldsymbol{\omega};M}(X_{i})d\Pi_{M}(\theta,\boldsymbol{\omega})}.
$$
\noindent
The first requirement to get consistency of estimators or posterior distributions is the identifiability of the model. We  use the following assumption.
\begin{assumption}[A1]
$\\$
$\bullet$ For all $j=1,\ldots,k$, $\theta_{j}^{\star}>0$.\\
$\bullet$ For all $c=1,2,3$, for all $j=1,\ldots,k$, $f^{\star}_{j,c}$ is continuous and almost surely positive and 
\begin{equation}  \label{moment:cond}\forall j_1, j_2 \leq k,  \quad 0 < \inf_x \frac{f_{j_1,c}^\star(x) }{f_{j_2,c}^\star(x)} \leq \sup_x \frac{f_{j_1,c}^\star(x) }{f_{j_2,c}^\star(x)}<  +\infty \end{equation}
$\bullet$  For all $c=1,2,3$, the measures $f^{\star}_{1,c}dx,\ldots,f^{\star}_{k,c}dx$ are linearly independent.
\end{assumption}

Note that the two first points in Assumption (A1) imply that for all $M$,  $(\theta^{\star},{\boldsymbol{\omega}}^{\star}_{M})$ lies in the interior of $\Theta\times \Omega_{M}$.
\\

\noindent
It is proved in Theorem 8 of \cite{allman2009identifiability} that under (A1) identifiability holds up to label switching, that is, if ${\cal T}_{k}$ is the set of permutations of $\{1,\ldots,k\}$,
$$
\forall \theta\in\Theta,\;\forall \mathbf{f}\in {\cal F}^{3k},\;g_{\theta,\mathbf{f}}=g_{\theta^{\star},\mathbf{f}^{\star}}\Longrightarrow \exists \sigma\in{\cal T}_{k}
\;{\text{such that}}\;
\prescript{\sigma}{}{\theta}=\theta^{\star},\;\prescript{\sigma}{}{\bf f}={\bf f}^{\star},
$$
where $\prescript{\sigma}{}{\theta} \in \Theta$ (resp. $\prescript{\sigma}{}{\bf f}\in\mathcal{F}^{3k}$, 
$\prescript{\sigma}{}{\theta}_{j}=\theta_{\sigma(j)}$, $\prescript{\sigma}{}{f}_{j,c}=f_{\sigma(j),c}$,  $c\in \{1,2,3\}$, $j \in \{1,,\ldots,k\}$) denotes the image of $\theta$ after permuting the labels using $\sigma$. 
This also implies  the identifiability  of model (\ref{eq:model2})  if the partition is refined enough. We also need the following assumption to ensure that all functions $f_{j,c;M}^{\star}$
 tend to $f^{\star}_{j,c}$ Lebesgue almost everywhere, where $f_{j,c;M}^{\star}$ is the function defined in (\ref{fomega}) with $\omega=(\omega^{\star}_{j,c,m})_{m}$.
 
\begin{assumption}[A2]
$\\$
$\bullet$ For all $M$, the sets $I_{m}$ in ${\cal I}_{M}$ are intervals with non empty interior.\\
$\bullet$ As $M$ tends to infinity, $\max_{1\leq m \leq M} |I_{m}|$ tends to $0$.
\end{assumption}

\noindent
Then, if (A1) and (A2) hold, for $M$ large enough, we have that for all  $c=1,2,3$,  the measures $f_{1,c;M}^{\star}dx,\ldots,f_{k,c;M}^{\star}dx$ are linearly independent.
We give a formal proof of this fact in Section \ref{subsec:proof:fisher}.
Thus, using again the identifiability result in \cite{allman2009identifiability}, under (A1) and (A2), for $M$ large enough,
\begin{multline*}
\forall \theta\in\Theta,\;\forall \boldsymbol{\omega}\in\Omega_{M},\;g_{\theta,\boldsymbol{\omega};M}=g_{\theta^{\star},\boldsymbol{\omega}^{\star}_{M};M}\\
\Longrightarrow
 \exists \sigma\in{\cal T}_{k}\;{\text{such that}}\;
\prescript{\sigma}{}{\theta}=\theta^{\star},\;\prescript{\sigma}{}{{\boldsymbol{\omega}}}={\boldsymbol{\omega}}^\star_M, \quad
\end{multline*}
where $\prescript{\sigma}{}{{\boldsymbol{\omega}}}\in\Omega_{M}$ and
$\prescript{\sigma}{}{\omega}_{j,c,m}=\omega_{\sigma(j),c,m}$, for all $m\in\{1,\ldots,M\}$, $c\in\{1,2,3\}$, $j\in\{1,\ldots,k\}$.

\subsection{Efficient influence functions and efficient Fisher informations}
\label{subsec:efficiency}

We  now study  the estimation of $\theta$ in model (\ref{eq:model1}) and in model (\ref{eq:model2}) from the semiparametric point of view, following Le Cam's theory. 
We  start with model (\ref{eq:model2}) which is easier to analyze since it is a parametric model. For any $M$, $g_{\theta,\boldsymbol{\omega};M}(\mathbf{x})$ is a polynomial function of the parameter $(\theta,\boldsymbol{\omega})$ and the model is differentiable in quadratic mean 
{ in the interior of $\Theta\times \Omega_{M}$.} Denote by
$S^{\star}_{M}=(S^{\star}_{\theta,M},S^{\star}_{\boldsymbol{\omega},M})$ the score function  for parameter $(\theta,\boldsymbol{\omega})$ at point $(\theta^{\star},\boldsymbol{\omega}^{\star}_{M})$  in model (\ref{eq:model2}). We have for $j=1,\ldots,k-1$
\begin{equation}
\label{eq:scoreM1}
\left(S^{\star}_{\theta,M}\right)_{j}=\frac{\prod_{c=1}^{3}f_{\omega^{\star}_{j,c;M}}-\prod_{c=1}^{3}f_{\omega^{\star}_{k,c;M}}}{g_{\theta^{\star},\boldsymbol{\omega}^{\star}_{M};M}},
\end{equation}
and for $j=1,\ldots,k$, $c=1,2,3$, $m=1,\ldots,M-1$
\begin{equation}
\label{eq:scoreM2}
\left(S^{\star}_{\boldsymbol{\omega},M}\right)_{j,c,m}=\frac{\theta_{j}^{\star}\left(\frac{\1_{I_{m}} (x_{c})}{|I_{m}|}-\frac{\1_{I_{M}} (x_{c})}{|I_{M}|}\right)\prod_{c'\neq c}f_{\omega^{\star}_{j,c';M}}}{g_{\theta^{\star},\boldsymbol{\omega}^{\star}_{M};M}}.
\end{equation}
Denote by $J_{M}$ the Fisher information, that is the variance of $S^{\star}_{M}(X)$:
$$
J_{M}=\E^{\star}\left[S^{\star}_{M}(X)S^{\star}_{M}(X)^{T}\right]
$$
Here, $\E^{\star}$ denotes expectation under $\PP^{\star}$, and $S^{\star}_{M}(X)^{T}$ is the transpose vector of $S^{\star}_{M}(X)$.\\

\noindent
When considering the question of efficient estimation of $\theta$ in the presence of a nuisance parameter, the relevant mathematical objects are the efficient influence function and the efficient Fisher information.  Recall that  the efficient score function is the projection of the score function with respect to parameter $\theta$ on the orthogonal subspace of the closure of the linear subspace spanned by the tangent set with respect to the nuisance parameter (that is the set of scores in parametric models regarding the nuisance parameter), see \cite{AadBook} or \cite{AadStFlour} for details.
The efficient Fisher information is the variance matrix of the efficient score function. For parametric models, a direct computation gives the result. If we partition the Fisher information $J_{M}$ according to the parameters $\theta$ and ${\boldsymbol{\omega}}$, that is
$$
[J_{M}]_{\theta,\theta}=\E^{\star}\left[S^{\star}_{\theta,M}(X)S^{\star}_{\theta,M}(X)^{T}\right],\;[J_{M}]_{\boldsymbol{\omega},\boldsymbol{\omega}}=\E^{\star}\left[S^{\star}_{\boldsymbol{\omega},M}(X)S^{\star}_{\boldsymbol{\omega},M}(X)^{T}\right],\;
$$
$$[J_{M}]_{\theta,\boldsymbol{\omega}}=\E^{\star}\left[S^{\star}_{\theta,M}(X)S^{\star}_{\boldsymbol{\omega},M}(X)^{T}\right],\;[J_{M}]_{\boldsymbol{\omega},\theta}=[J_{M}]_{\theta,\boldsymbol{\omega}}^{T},\;
$$
we get that, in model  (\ref{eq:model2}), 
if we denote $\tilde{\psi}_{M}$  the efficient score function for the estimation of $\theta$,
$$
\tilde{\psi}_{M}=S^{\star}_{\theta,M}-[J_{M}]_{\theta,\boldsymbol{\omega}}([J_{M}]_{\boldsymbol{\omega},\boldsymbol{\omega}})^{-1}S^{\star}_{\boldsymbol{\omega},M},
$$
and the efficient Fisher information $\tilde{J}_{M}$ is the $(k-1)\times (k-1)$-matrix given by
$$
\tilde{J}_{M}=[J_{M}]_{\theta,\theta}-[J_{M}]_{\theta,\boldsymbol{\omega}}([J_{M}]_{\boldsymbol{\omega},\boldsymbol{\omega}})^{-1}[J_{M}]_{\theta,\boldsymbol{\omega}}^{T}.
$$
To discuss efficiency of estimators, invertibility of the efficient Fisher information is needed. Spectral methods have been proposed recently to get estimators in model (\ref{eq:model2}), see \cite{kakade14}. It is possible to obtain upper bounds  of their local maximum quadratic risk with rate $n^{-1/2}$, which as a consequence excludes the possibility that the efficient Fisher information be singular. This is stated in Proposition \ref{prop:fisher} below and proved in Section \ref{subsec:proof:fisher}.
\begin{prop}
\label{prop:fisher}
Assume (A1) and (A2). Then, for large enough $M$, $\tilde{J}_M$ is non singular.
\end{prop}
In the context of mixture models, all asymptotic results are given up to label switching. We define here formally what we mean by `up to label switching' for frequentist efficiency results with Equation \eqref{eq:effiM_def_label_switching}  and Bayesian efficiency results with Equation \eqref{eq:bvmM_def_label_switching}.

By Proposition \ref{prop:fisher}, if (A1) and (A2) hold, for large enough $M$, $\tilde{J}_{M}$ is non singular, and an estimator $\widehat{\theta}$ is asymptotically a regular efficient estimator of $\theta^{\star}$ if and only if
\begin{equation}
\label{eq:effiM}
\sqrt{n}\left(\widehat{\theta}-\theta^{\star}\right)=\frac{\tilde{J}_{M}^{-1}}{\sqrt{n}}\sum_{i=1}^{n}\tilde{\psi}_{M}\left(X_{i}\right)+o_{\PP^{\star}}(1), \quad \text{up to label switching,}
\end{equation}
which formally means that there exists a sequence $(\sigma_n)_n$ belonging to $\mathcal{T}_k$ such that
\begin{equation}
\label{eq:effiM_def_label_switching}
\sqrt{n}\left(\prescript{\sigma_n}{}{\widehat{\theta}}-\theta^{\star}\right)=\frac{\tilde{J}_{M}^{-1}}{\sqrt{n}}\sum_{i=1}^{n}\tilde{\psi}_{M}\left(X_{i}\right)+o_{\PP^{\star}}(1).
\end{equation}
To get an asymptotically  regular efficient estimator, one may for instance use the MLE  $\widehat{\theta}_{M}$ (see the beginning of the proof of Theorem \ref{th:FreeLunch}). One may also apply a one step improvement (see Section 5.7
 in \cite{AadBook}) of a preliminary spectral estimator, such as the one described in \cite{kakade14}. 
 
\noindent
In the Bayesian context, Bernstein-von Mises Theorem holds for large enough $M$ if the prior has a positive density in the neighborhood of $(\theta^{\star},\boldsymbol{\omega}^{\star}_{M})$, see Theorem 10.1 in \cite{AadBook}. 
That is,  
if $\|\cdot\|_{TV}$ denotes the total variation distance, 
with $\Pi_{M,\theta}$ the marginal distribution on the parameter $\theta$,
\begin{equation}
\label{eq:bvmM}
\left\|\Pi_{M,\theta}\left(\cdot \vert X_{1},\ldots,X_{n}\right)-{\cal N}\left(\widehat{\theta}; \frac{\tilde{J}_{M}^{-1}}{n}   \right)\right\|_{TV}=o_{\PP^{\star}}(1), \quad \text{up to label switching,}
\end{equation}
where $\widehat{\theta}$ verifies Equation \eqref{eq:effiM}, 
which formally means that
\begin{equation}
\label{eq:bvmM_def_label_switching}
\sup_{A\subset \Theta} \left|\Pi_{M,\theta}\big(\exists \sigma \in \mathcal{T}_k: ~ \prescript{\sigma}{}{\theta} \in A
 \big\vert X_{1},\ldots,X_{n}\big)
 -
 {\cal N}\left(\prescript{\sigma_n}{}{\widehat{\theta}}; \frac{\tilde{J}_{M}^{-1}}{n}   \right) (A) \right|=o_{\PP^{\star}}(1),
\end{equation}
where $(\sigma_n)$ and $\widehat{\theta}$ satisfy Equation \eqref{eq:effiM_def_label_switching}.

\noindent
A naive heuristic idea is that, when using the $Y_{i}$'s as summaries of the $X_{i}$'s, one has less information, but more and more if the partition  ${\cal I}_{M}$ is refined. Thus, the efficient Fisher information should grow when partitions ${\cal I}_{M}$ are refined. The following proposition is proved in Section \ref{subsec:proof:FisherGrow}. 
\begin{prop}
\label{prop:FisherGrow}
Let ${\cal I}_{M_{1}}$ be a coarser partition than ${\cal I}_{M_{2}}$, that is such that for any $I\in {\cal I}_{M_{1}}$, there exists $A\subset {\cal I}_{M_{2}}$ such that $I=\cup_{I'\in A} I'$. Then
$$
\tilde{J}_{M_{2}} \geq \tilde{J}_{M_{1}}
$$
in which`` $\geq$" denotes the partial order between symmetric matrices.
\end{prop}
\noindent
Thus, it is of interest to let the partitions grow so that one reaches the largest efficient Fisher information.
\\

\noindent
Let us now come back to model (\ref{eq:model1}). Let, for $j=1,\ldots,k$, $c=1,2,3$, ${\cal H}_{j,c}$ be the subset of functions $h$ in 
$ L^{2}(f^{\star}_{j,c}dx)$ 
such that $\int h f^{\star}_{j,c}dx=0$. Then the tangent set for ${\mathbf{f}}$ at point $(\theta^{\star},\mathbf{f}^{\star})$ is
the subspace ${\dot{\cal P}}$ of  $L^{2}(g_{\theta^{\star},\mathbf{f}^{\star}}(\mathbf{x})d\mathbf{x})$ 
spanned by the functions
$$
\mathbf{x} \mapsto \frac{h(x_{c})\prod_{c'=1}^3 f^{\star}_{j,c'}(x_{c'})}{g_{\theta^{\star},\mathbf{f}^{\star}}(\mathbf{x})},\;h\in{{\cal H}_{j,c}},\;j=1,\ldots,k,\;c=1,2,3.
$$
Notice that for each  $j=1,\ldots,k$ and $c=1,2,3$, ${\cal H}_{j,c}$ is a closed linear subset of $ L^{2}(f^{\star}_{j,c}dx)$ and that ${\dot{\cal P}}$ is a closed linear subset of  $L^{2}(g_{\theta^{\star},\mathbf{f}^{\star}}(\mathbf{x})d\mathbf{x})$.
The efficient score function  $\tilde{\psi}$ for the estimation of $\theta$ in the semiparametric model (\ref{eq:model1}) is given, for $j=1,\ldots,k-1$, by
\begin{equation}
\label{eq:effiscore}
\tilde{\psi}_{j}=\left(S^{\star}_{\theta}\right)_{j}-{\mathbb{A}} \left(S^{\star}_{\theta}\right)_{j},\;\left(S^{\star}_{\theta}\right)_{j}=\frac{\prod_{c=1}^{3}f^{\star}_{j,c}-\prod_{c=1}^{3}f^{\star}_{k,c}}{g_{\theta^{\star},\mathbf{f}^{\star}}},
\end{equation}
with $\mathbb{A}$ the orthogonal projection onto 
${\dot{\cal P}}$ in  $L^{2}(g_{\theta^{\star},\mathbf{f}^{\star}}(\mathbf{x})d\mathbf{x})$. Then,
the efficient Fisher information  $\tilde{J}$ is the variance matrix of $\tilde{\psi}$.\\
If $\tilde{J}$ is non singular,
an estimator $\widehat{\theta}$ is asymptotically a regular efficient estimator of $\theta^{\star}$ if and only if
\begin{equation}
\label{eq:effi}
\sqrt{n}\left(\widehat{\theta}-\theta^{\star}\right)=\frac{\tilde{J}^{-1}}{\sqrt{n}}\sum_{i=1}^{n}\tilde{\psi}\left(X_{i}\right)+o_{\PP^{\star}}(1), \quad \text{up to label switching}
\end{equation}
and a Bayesian method using a nonparametric prior $\Pi$ satisfies a semiparametric Bernstein-von Mises theorem if, with $\Pi_{\theta}$ the marginal distribution on the parameter $\theta$,
\begin{equation}
\label{eq:bvm}
\left\|\Pi_{\theta}\left(\cdot \vert X_{1},\ldots,X_{n}\right)-{\cal N}\left(\widehat{\theta}; \frac{\tilde{J}^{-1}}{n}   \right)\right\|_{TV}=o_{\PP^{\star}}(1), \quad \text{up to label switching}
\end{equation}
for a $\widehat{\theta}$ satisfying (\ref{eq:effi}).
\noindent

\subsection{General result}
\label{subsec:gene}

When the size of the bins in the partition decreases,
 we expect that the efficient score functions in (\ref{eq:model2}) are good approximations of the efficient score functions in (\ref{eq:model1}) so that asymptotically efficient estimators in model (\ref{eq:model2}) become efficient estimators in model (\ref{eq:model1}). This is what Theorem \ref{th:FreeLunch} below states, under the following additional assumption :
\begin{assumption}[A3]
For all $M$ large enough, ${\cal I}_{M}$ is a coarser partition than ${\cal I}_{M +1}$
\end{assumption}

We first obtain:
\begin{lemma}
\label{lem:convscore}
Under Assumptions (A1), (A2) and (A3),
 the sequence of score functions $({\tilde{\psi}}_{M})_{M}$ converges in $L^{2}(g_{\theta^{\star},\mathbf{f}^\star}d\mathbf{x})$ to the score function $\tilde{\psi}$,  and the sequence of efficient Fisher informations $({\tilde{J}}_{M})_{M}$ converges to the efficient Fisher information matrix $\tilde{J}$, 
 which is non singular. 
\end{lemma}
The invertibility of $\tilde{J}$ is a consequence of Proposition \ref{prop:fisher}, Proposition \ref{prop:FisherGrow}, and the convergence of $\tilde J_M$ to $\tilde J$.

We are now ready to state Theorem \ref{th:FreeLunch}.
\begin{theo}
\label{th:FreeLunch}
Under Assumptions (A1), (A2),  and (A3), 
 there exists a sequence $M_{n}$ tending to infinity sufficiently slowly such that the MLE $\widehat{\theta}_{M_{n}}$ is asymptotically a regular efficient estimator of $\theta^{\star}$ and satisfies 
\begin{equation}\label{MLE:eff}
\sqrt{n}\left(\widehat{\theta}_{M_{n}}-\theta^{\star}\right)=\frac{\tilde{J}^{-1}}{\sqrt{n}}\sum_{i=1}^{n}\tilde{\psi}\left(X_{i}\right)+o_{\PP^{\star}}(1), \quad \text{up to label switching}.
\end{equation}
Moreover, under the same assumptions and if  for all $M$, the prior $\Pi_{M}$   has a positive  and continuous density at $(\theta^{\star},\boldsymbol{\omega}^{\star}_{M})$, then there exists a sequence $L_{n}$ tending to infinity sufficiently slowly such that 
\begin{multline}\label{bvm:eff}
\left\|\Pi_{L_{n},\theta}\left(\cdot \vert X_{1},\ldots,X_{n}\right)-{\cal N}\left(\theta^{\star}+\frac{\tilde{J}^{-1}}{n}\sum_{i=1}^{n}\tilde{\psi}\left(X_{i}\right); \frac{\tilde{J}^{-1}}{n}   \right)\right\|_{TV} \!\!\!\!=o_{\PP^{\star}}(1), \\ \text{up to label switching}.
\end{multline}
\end{theo}

Note that, in Theorem \ref{th:FreeLunch}, any sequence $M_n'\leq M_n$ going to infinity also satisfies \eqref{MLE:eff} and similarly \eqref{bvm:eff} holds for any sequence $L_n'\leq L_n$ going to infinity. 

\begin{proof}
As explained in Section \ref{subsec:notations}, model \eqref{eq:model2} is the correct model associated with the observations made of the counts per bins and under Assumption (A1) it is a regular model in the neighborhood of the true parameter. 
 Also, using the identifiability of the model and the trick given in  \cite{AadBook} p. 63,  we get consistency of the MLE. Thus, it is possible to apply Theorem 5.39 in  \cite{AadBook} to get that for each $M$, the MLE  $\widehat{\theta}_{M}$ is regular and asymptotically efficient, that is
$$
\sqrt{n}\left(
\prescript{\sigma_{n,M}}{}
{\widehat{\theta}}_{M}-\theta^{\star}\right)=\frac{\tilde{J}_{M}^{-1}}{\sqrt{n}}\sum_{i=1}^{n}\tilde{\psi}_{M}\left(X_{i}\right)+R_{n}(M),
$$
where for each $M$, $\sigma_{n,M}$ is a sequence of permutations in ${\cal T}_{k}$, and 
$(R_{n}(M))_{n\geq 1}$ is a sequence of random vectors converging to $0$ in $\PP^{\star}$-probability as $n$ tends to infinity. Therefore, there exists  a sequence $M_{n}$ tending to infinity sufficiently slowly so that, as  $n$ tends to infinity, $R_{n}(M_{n})$ tends to $0$  in $\PP^{\star}$-probability. Now,
\begin{eqnarray*}
\frac{\tilde{J}_{M_{n}}^{-1}}{\sqrt{n}}\sum_{i=1}^{n}\tilde{\psi}_{M_{n}}\left(X_{i}\right)&=&\frac{\tilde{J}^{-1}}{\sqrt{n}}\sum_{i=1}^{n}\tilde{\psi}\left(X_{i}\right)+\frac{\tilde{J}_{M_{n}}^{-1}-\tilde{J}^{-1}}{\sqrt{n}}\sum_{i=1}^{n}\tilde{\psi}\left(X_{i}\right)
\\
& &+\frac{\tilde{J}_{M_{n}}^{-1}}{\sqrt{n}}\sum_{i=1}^{n}(\tilde{\psi}_{M_{n}}-\tilde{\psi})\left(X_{i}\right)
\\
&=&\frac{\tilde{J}^{-1}}{\sqrt{n}}\sum_{i=1}^{n}\tilde{\psi}\left(X_{i}\right)+o_{\PP^{\star}}(1)
\end{eqnarray*}
since, by Lemma \ref{lem:convscore}, $\E^{\star}\|\frac{1}{\sqrt{n}}\sum_{i=1}^{n}(\tilde{\psi}_{M_{n}}-\tilde{\psi})\left(X_{i}\right)\|^{2}=\|\tilde{\psi}_{M_{n}}-\tilde{\psi}\|_{L^{2}(g_{\theta^{\star},\mathbf{f}^{\star}}(\mathbf{x})d\mathbf{x})}^{2}$ tends to $0$ as $n$ tends to infinity and $(\tilde{J}_{M_n})^{-1}$ converges to $(\tilde{J})^{-1}$ as $n$ tends to infinity,
so that the first part of the theorem is proved.\\
On the Bayesian side, for all $M$, there exists  a sequence $V_{n}(M)$ of random vectors converging to $0$  in $\PP^{\star}$-probability as $n$ tends to infinity such that
$$
\sup_{A\subset \Theta} \left|\Pi_{M,\theta}\big( \exists \sigma\in \mathcal{T}_k: ~ \prescript{\sigma}{}{\theta} \in A \big\vert X_{1},\ldots,X_{n}\big)
-{\cal N}\left(\prescript{\sigma_{n,M}}{}{\widehat{\theta}}_{M}; \frac{\tilde{J}_{M}^{-1}}{n}   \right)(A)\right|=V_{n}(M).
$$
Arguing as previously, there exists  a sequence $L_{n}$ tending to infinity sufficiently slowly so that, as  $n$ tends to infinity, both $V_{n}(L_{n})$  and $R_{n}(L_{n})$ tend to $0$  in $\PP^{\star}$-probability. 
Using the fact that the total variation distance is invariant through one-to-one transformations we get
\begin{multline*}
\left\|{\cal N}\left(\prescript{\sigma_{n,M}}{}{\widehat{\theta}}_{M}; \frac{\tilde{J}_{M}^{-1}}{n}\right)-{\cal N}\left(\theta^{\star}+\frac{\tilde{J}^{-1}}{n}\sum_{i=1}^{n}\tilde{\psi}\left(X_{i}\right); \frac{\tilde{J}^{-1}}{n}   \right)\right\|_{TV}\\
= 
\left\|{\cal N}\left(\sqrt{n}\left(\prescript{\sigma_{n,M}}{}{\widehat{\theta}}_{M}-\theta^{\star}\right)-
\frac{\tilde{J}^{-1}}{\sqrt{n}}\sum_{i=1}^{n}\tilde{\psi} \left(X_{i}\right); \tilde{J}_{M}^{-1}\right)-{\cal N}\left(0;\tilde{J}^{-1}  \right) \right\|_{TV}\\
= 
\left\|{\cal N}\left(\tilde{J}_{M}^{1/2}[\sqrt{n}\left(\prescript{\sigma_{n,M}}{}{\widehat{\theta}}_{M}-\theta^{\star}\right)-
\frac{\tilde{J}^{-1}}{\sqrt{n}}\sum_{i=1}^{n}\tilde{\psi}\left(X_{i}\right)]; Id\right)-{\cal N}\left(0;\tilde{J}_{M}\tilde{J}^{-1}  \right) \right\|_{TV}\\
\leq 
\left\|{\cal N}\left(\tilde{J}_{M}^{1/2}[\sqrt{n}\left(\prescript{\sigma_{n,M}}{}{\widehat{\theta}}_{M}-\theta^{\star}\right)-
\frac{\tilde{J}^{-1}}{\sqrt{n}}\sum_{i=1}^{n}\tilde{\psi}\left(X_{i}\right)]; Id\right)-{\cal N}\left(0;Id  \right) \right\|_{TV}\\
+\left\|{\cal N}\left(0,Id\right)-{\cal N}\left(0;\tilde{J}_{M}\tilde{J}^{-1}  \right)\right\|_{TV}.
\end{multline*}
But for vectors in $m\in {\mathbb R}^{k-1}$ and symmetric positive $(k-1)\times (k-1)$ matrices $\Sigma$ we have
$$
\left\|{\cal N}\left(m,Id\right)-{\cal N}\left(0;Id  \right)\right\|_{TV} \leq \|m\|
$$
and
\begin{multline*}
\left\|{\cal N}\left(0,Id\right)-{\cal N}\left(0;\Sigma \right)\right\|_{TV}\leq \PP \left(\|\Sigma^{1/2}U\|^{2}-\|U\|^{2} \geq \log [det (\Sigma)] \right)\\-\PP \left(\|U\|^{2}-\|\Sigma^{-1/2}U\|^{2} \geq \log [det (\Sigma)] \right)
\end{multline*}
where $U\sim {\cal N}\left(0,Id\right)$.
Thus the last part of the theorem follows from the triangular inequality and the fact that using Lemma \ref{lem:convscore},  as n tends to infinity, $\tilde{J}_{L_n}\tilde{J}^{-1}$ tends to $Id$, the identity matrix, and $V_{n}(L_{n})$  and $R_{n}(L_{n})$ tend to $0$  in $\PP^{\star}$-probability.
\end{proof}

\section{Model selection}
\label{sec:modelselection}

In Theorem \ref{th:FreeLunch}, we prove the existence of some increasing partition leading to efficiency. In this section, we propose a  practical method to choose a partition when the number of observations $n$ is fixed. In Section \ref{se:reasons_model_selection} we prove that one has to take care to choose not too large $M_{n}$'s since sequences $(M_{n})_{n}$ tending too quickly to infinity lead to inconsistent estimators. In Section \ref{se:criterion_model_selection}, we propose a cross-validation method to estimate the oracle value $M_{n}^{\star}$ minimizing the unknown risk as a function of $M$:
$M\mapsto \E^{\star} \left[\|\widehat{\theta}_{M}-\theta^{\star}\|^{2} \right]$ (up to label switching). 

\subsection{Behaviour of the MLE  as $M$ increases}\label{se:reasons_model_selection}

We first explain why the choice of the model is important. We have seen in Proposition~\ref{prop:FisherGrow} that for a sequence of increasing partitions, the efficient matrix is non decreasing. The question is then: can we take any sequence tending to infinity wih $n$? Or, for a fixed $n$, can we take any 
$M$ arbitrarily large? As is illustrated in Figure 2, we see that if $M$ is too large  (or equivalently if $M_n$ goes to infinity too fast) the MLE (or the Bayesian procedure) is biased. 

In Proposition \ref{prop:limit_thetaM}, we give the limit of the MLE when the number  $n$ of observations is fixed but $M$ tends to infinity.

\begin{prop}\label{prop:limit_thetaM}
Under Assumptions (A1) and  (A2). For almost all observations $X_1, \dots, X_n$, $\widehat{\theta}_M(X_1, \dots, X_n)$ tends to 
 \[\underline{\theta}_n=
 (\underbrace{\lfloor n/k \rfloor/n, \dots , \lfloor n/k \rfloor/n}_{r:=n-k\lfloor n/k \rfloor }
 , 
 \underbrace{\lceil n/k \rceil/n, \dots, \lceil n/k \rceil/n}_{k-r })
 \]
 up to label switching, when M tends to infinity.
\end{prop}

Proposition \ref{prop:limit_thetaM} is proved in Section \ref{sec:proofProp4}.

Using Proposition \ref{prop:limit_thetaM}, we can deduce a constraint on sequences $M_n$ leading to consistent estimation of $\theta^{\star}$, depending on the considered sequence of partitions $({\cal I}_M)_{M\in {\cal M}}$, which may give an upper bound on sequences $M_n$ leading to efficiency. We believe that this constraint is very conservative and leads to very conservative bounds. Corollary \ref{co:bound_for M_n} below is proved in Section \ref{sec:proof_cor_lim_Mn}.

\begin{cor}\label{co:bound_for M_n}
Assume  that (A1), (A2) and (A3) hold.
If $\widehat{\theta}_{M_n}$ tends to $\theta^{\star}$ in probability
and if $\theta^{\star}$ is different from $(1/k, \dots ,1/k)$,
then there exists $N>0$ and a constant $C>0$ such that for all $n\geq N$,
\[n^2 \left(\max_{m\leq M_n} |I_m|\right)^2 M_n \geq C.
\]
In  particular, if  there exists $0<C_1\leq C_2$ such that for all $n\in \mathbb{N}$ and $1\leq m \leq M_n$,
\begin{equation}\label{hyp:bornes_sur_taille_segments}
\frac{C_1}{M_n} \leq |I_m | \leq \frac{C_2}{M_n}
\end{equation}
then 
there exists a constant $C>0$ such that,
\[
M_n\leq C n^2
.\]
\end{cor}

Note that Assumption \eqref{hyp:bornes_sur_taille_segments} holds as soon as the partition is regular, and in particular for the dyadic regular partitions
which forms an embedded sequence of partitions where  ${\cal M}= \{2^p, p \in \mathbb{N}^\star \}$ and for all $M\in {\cal M}$ $I_m=[(m-1)/M,m/M)$ for all $m < M$, $I_M=[(M-1)/M,1]$.

\subsection{Criterion for model selection}\label{se:criterion_model_selection}

In this section, we propose a criterion to choose the partition when $n$ is fixed. This criterion can be used to choose the size $M$ of a family of partitions but also to choose between two families of partition.
For each dataset, we can compute the MLE or the posterior mean or other Bayesian estimators under model \eqref{eq:model2} with partition $\cal I$. We thus shall index all our  estimators by $\cal I$. Note that the results of this section are valid for any family of estimators $(\tilde{\theta}_{\cal I})$ and not only for the MLE $(\widehat{\theta}_{\cal I})$. But we illustrate our results using the MLE.
\\

Proposition \ref{prop:limit_thetaM} and Corollary \ref{co:bound_for M_n} show the necessity to choose an appropriate partition among a collection of partitions ${\cal I}_M$, $M \in {\cal M}$. To choose the partition we need a criterion. Since the aim is to get efficient estimators, we choose the quadratic risk as the criterion to minimize. We thus want to minimize over all possible partitions
\begin{equation}
R_n({\cal I})=\E^{\star} \left[\|\tilde{\theta}_{\cal I}(X_{1:n})-\theta^{\star}\|_{\mathcal{T}_k}^{2} \right],
\end{equation}
where $X_{1:n}=(X_i)_{i\leq n}$
and for all $\theta$, $\tilde{\theta}\in \Theta$,
\begin{equation}
\lVert\theta - \tilde{\theta} \rVert_{\mathcal{T}_k}
=\min_{\sigma \in \mathcal{T}_k} \lVert \prescript{\sigma}{}{\theta} - \tilde{\theta} \rVert_2 .
 \end{equation}
As usual, this criterion cannot be computed in practice (since we do not know $\theta^\star$) and we need for each partition $\cal I$ some estimator $C(\cal I)$ of $R_n({\cal I})$.

We want to emphasize here that the choice of the criterion for this problem is not easy. 
Indeed, the  quadratic risk $R_n({\cal I})$ cannot be written as the expectation of an excess loss expressed thanks to a contrast function, i.e. in the form $\E^{\star}\left[\E^{\star} \left[ \gamma(\tilde{\theta}(X_{1:n}),X) - \gamma(\theta^\star,X) | X_{1:n}\right]\right]$, where $\gamma: ~ \Theta \times \mathcal{X} \to [0,+\infty) $. Yet, the latter is the framework of most theoretical results in model selection, see \cite{MR2602303} or \cite{massart03} for instance.
Moreover  decomposing the quadratic risk  as an approximation error plus an estimation error as explained in \cite{MR2602303}:
\[
R_n({\cal I})= 
\underbrace{
\inf_{\theta \in \Theta_{\cal I}} \left\|  \theta - \theta^\star \right\|_{\mathcal{T}_k}^2
}_{\text{approximation error}}
 + ~
  \underbrace{ 
 R_n({\cal I}) - \inf_{\theta \in \Theta_{\cal I} } \left\|  \theta - \theta^\star \right\|_{\mathcal{T}_k}^2
  }_{\text{estimation error}}
, \quad \text{ where } \Theta_{\cal I}=\Theta,\]
we see that the approximation error is always zero in our model (and not decreasing as often  when the complexity of the models increases). Hence 
 \begin{align}\label{eq:risk_bias_var}
R_n({\cal I}_M) &=
 \underbrace{Var^{\star} \left[\tilde{\theta}_{{\cal I}_M}(X_{1:n})\right]}_{\text{variance}}
+ 
\underbrace{\left\|\E^{\star} \left[\tilde{\theta}_{{\cal I}_M}(X_{1:n})\right]-\theta^{\star}\right\|_{\mathcal{T}_k}^{2}}_{\text{bias}}
\end{align}
where $Var^{\star}(.)$ is to be understood as the trace of the variance matrix. Here the bias is only an estimation bias and not a model mispecification bias. 

In the case of the  MLE, using Theorem \ref{th:FreeLunch}, for all fixed $M$ (large enough), the regularity of the mixture of these multivariate distributions implies that the bias is $O(1/n) $ and the variance converges to the inverse Fisher information matrix so that 
$$R_n({\cal I}_M) = Var^{\star} \left[
\widehat{\theta}_{{\cal I}_M}(X_{1:n})\right] + O(1/n) = tr\left( \tilde J_M^{-1}\right) + O(1/n) $$
and if $M \geq M_\epsilon$ so that $\|\tilde J_{M_\epsilon}^{-1} -  \tilde J^{-1} \| \leq \epsilon$ we obtain that for $n$ large enough
$$\mbox{tr}\left( \tilde J^{-1}\right) - 2\epsilon \leq R_n({\cal I}_M) \leq \mbox{tr}\left( \tilde  J^{-1}\right) + 2\epsilon.$$
Minimizing $R_n({\cal I}_M)$ therefore corresponds to choosing $M$ such that $\tilde J_M$ is close enough to $\tilde J$ (i.e. $M$ large enough)  while not deteriorating too much the approximation of $R_n({\cal I}_M)$ by $\mbox{tr}(\tilde J_{M}^{-1})$ (i.e. $M$ not too large).

Because the approximation error is always zero we cannot apply the usual methods and we   use instead a variant of the cross-validation technique.

Consider a partition of $\{1, \cdots, n \}$ in the form $(B_b, B_{-b}, b \leq b_n )$, in other words the partition is made of $2 \times b_n$ subsets of $ \{1, \cdots , n\}$. By definition  $B_{b_1} \cap B_{-b_2} =\emptyset$ for all $b_1, b_2 \leq b_n$. Because an arbitrary estimator, e.g. the MLE, based on any finite sample size is not unbiased, 
the following naive estimator of the risk is not appropriate:
$$
C_{CV1}({\cal I})=\frac{1}{2b_n}\sum_{b=1}^{b_n}\|\tilde{\theta}_{{\cal I}}(X_{B_b})-\tilde{\theta}_{{\cal I}}(X_{B_{-b}})\|_{\mathcal{T}_k}^{2}. 
$$
This can be seen by
decomposing the risk $R_n({\cal I})$  as in Equation \eqref{eq:risk_bias_var}
 and by computing the expectation of $C_{CV1}({\cal I})$ in  the case where the sizes of $B_b$, $B_{-b}$, $b \leq b_n$, are all equal,
\[
\E^{\star} \left[
C_{CV1}({\cal I})
\right]
=
Var^{\star} \left[ \tilde{\theta}_{\cal I}(X_{B_b}) \right].
\]
Then, the criterion $C_{CV1}({\cal I})$ do not capture the bias of the estimator $\tilde{\theta}_{\cal I}$. 

In the case of the MLE,
  using Proposition \ref{prop:limit_thetaM}, $C_{CV1}({\cal I})$ is tending to $0$ when $\max_m |I_m|$ tends to $0$. So that minimizing this criterion leads to choosing a partition $\widehat{{\cal I}}_n \in \argmin_{\cal I} C_{CV1}({\cal I})$ which has a large number of sets and so $\widehat{\theta}_{\widehat{{\cal I}}_n}(X_{1:n}) $ may be close to  $(1/k, \dots, 1/k)$ and then may not even be consistent. 
As an illustration, see  Figure \ref{fig:risk_biais_var} where  
$R_n({\cal I})$, 
$Var^{\star} \left[\widehat{\theta}_{\cal I}(X_{1:n})\right]$ and 
$\left\|\E^{\star} \left[\widehat{\theta}_{\cal I}(X_{1:n})\right]-\theta^{\star}\right\|_{\mathcal{T}_k}^{2}$ 
 are plotted as a function of $M$, for three simulation sets and various values of the sample size $n$, see Section \ref{sec:simu} for more details. It is quite clear from these plots that the variances remain either almost constant with $M$ or tend to decrease, while the bias  increases with $M$ and becomes dominant as $M$ becomes larger. As a result $R_n({\cal I})$ tends to first decrease and then increase as $M$ increases. 
 
 To address the bad behaviour of $C_{CV1}({\cal I})$, we use an idea of \cite{VdLaan}.  Choose a fixed base partition ${\cal I}_{0}$ with a small number of bins (although large enough to allow for identifiability). Then compute 
$$
C_{CV}({\cal I})=\frac{1}{ b_n}\sum_{b=1}^{b_n}\|\tilde{\theta}_{{\cal I}}(X_{B_b})-\tilde{\theta}_{{\cal I}_0} (X_{B_{-b}})\|_{\mathcal{T}_k}^{2}. 
$$
Ideally we would like to use a perfectly unbiased  estimator $\tilde{\theta}$ in the place of $\tilde{\theta}_{{\cal I}_0} (X_{B_{-b}})$,
see Assumption \ref{hyp:theta0_unbiased} used in Theorem \ref{th:construction_Mn} and Proposition \ref{prop:alm_oracle_ineq}.   We discuss the choice of $\tilde{\theta}_{{\cal I}_0} (X_{B_{-b}})$ at the end of the section.

Figure \ref{fig:pattern_criteria} gives an idea of the behaviour of $C_{CV}(  \cdot ) $ and $C_{CV1}(  \cdot )$ using the MLE. It shows in particular that in our simulation study $C_{CV}(  \cdot )$ follows the same behaviour as $R_n(  \cdot )$, contrarywise to $C_{CV1}(  \cdot )$. More details are given in Section \ref{sec:simu}.

We now provide some theoretical results on the behaviour of the minimizer of $C_{CV}(  \cdot ) $ over a finite family of candidate partitions $\mathcal M_n$ compared to the minimizer of $R_{a_n}(\cdot )$ over the same family. 
Let $m_n = \#{\cal M}_n $ be the number of candidate partitions. 
We consider the following set of assumptions:
\Needspace{3\baselineskip}
\begin{assumption}[A5]\label{hyp:A5}
\leavevmode
\begin{enumerate}[label=(A5.\arabic*),leftmargin=1.5cm]
\item\label{hyp:size_samplin_set_ind} $B_b$, $B_{-b}$, $b\leq b_n$ are disjoint sets of equal size
\[\#B_b=\#B_{-b}=a_n, \text{~ for all } b\leq b_n
\]
\item \label{hyp:theta0_unbiased} $\tilde{\theta}_{{\cal I}_0}(X_{B_{-b}})$ is not biased i.e.
$\mathbb{E}^\star[ \tilde{\theta}_{{\cal I}_0}
{(X_{B_{-b}})}]=\theta^\star$.
\end{enumerate}
\end{assumption}
We obtain the following oracle  inequality.

\begin{theo}\label{th:construction_Mn}
Suppose Assumption (A5). For any sequences $0<\epsilon_n, \delta_n <1 $, with probability greater than 
\[1-2m_n \exp{\left(-2b_n \left(\epsilon_n \inf_{ {\cal I} \in {\cal M}_n } R_{a_n}({\cal I}) + \delta_n\right)^2\right)}
,\]
we have
\begin{equation}\label{eq:ineq_selection_model}
R_{a_n}(\widehat{{\cal I}}_n) \leq \frac{1+ \epsilon_n}{1- \epsilon_n} \inf_{ {\cal I} \in {\cal M}_n } R_{a_n}({\cal I}) + \frac{ 2 \delta_n }{ 1 -\epsilon_n },
\end{equation}
where $\widehat{{\cal I}}_n\in \argmin_{{\cal I} \in {\cal M}_n} C_{CV}({\cal I})$. 
\end{theo}

As a consequence of Theorem \ref{th:construction_Mn}, the following Proposition holds. Recall that $n = 2 b_n a_n$. 

\begin{prop}\label{prop:alm_oracle_ineq}
Assume (A5). If $b_n \gtrsim n^{2/3} \log^2(n)$, $a_n \lesssim n^{1/3}/(\log^2(n))$, and  $m_n\leq C_\alpha n^\alpha$, for some $C_\alpha >0$ and $\alpha \geq 0$, 
then
\[
\E^{\star} \left[ a_n R_{a_n}(\widehat{{\cal I}}_n)\right]\leq
\inf_{ {\cal I} \in {\cal M}_n } a_n R_{a_n}({\cal I}) + o(1 ),
\]
where $\widehat{{\cal I}}_n\in \argmin_{{\cal I} \in {\cal M}_n} C_{CV}({\cal I})$. 
\end{prop}

Note that for each ${\cal I}$, $R_{a_n}({\cal I})$ is of order of magnitude $1/a_n$ so that the main term in the upper bound of Proposition \ref{prop:alm_oracle_ineq} is 
$\inf_{ {\cal I} \in {\cal M}_n } a_n R_{a_n}({\cal I})$. Note also that this is an exact oracle inequality (with constant $1$).

In Theorem \ref{th:construction_Mn} and Proposition \ref{prop:alm_oracle_ineq}, $\widehat{{\cal I}}_n$ is built on $n$ observations while the risk is associated with $a_n<n$ observations. This leads to a conservative choice of $\widehat{{\cal I}}_n$, i.e. we may choose a sequence $\widehat{{\cal I}}_n$ (optimal with $a_n$ observations) increasing more slowly than the optimal one (with $n$ observation). We think however that this conservative choice should not change the  good  behaviour of $\widehat{\theta}_{\widehat{{\cal I}}_n}$, since Theorem \ref{th:FreeLunch} implies that any sequence of partitions which grows  slowly enough to infinity leads to an efficient estimator. Hence, once the sequence $M_n$ growing to infinity is chosen, then any other sequence growing to infinity more slowly also leads to an efficient estimator.\\

In Proposition \ref{prop:alm_oracle_ineq} and Theorem \ref{th:construction_Mn}, the reference point estimate $\tilde{\theta}_{{\cal I}_0} (X_{B_{-b}})$ is assumed to be unbiased. This is a strong assumption, which is not exactly satisfied in our simulation study. To consider a reasonable approximation of it, $\tilde{\theta}_{{\cal I}_0} (X_{B_{-b}})$ is chosen as the MLE associated with a partition with a small number of bins.
 Recall that  the maximum likelihood estimator  is asymptotically unbiased and for a fixed $M$, the   bias of the MLE  for the whole parameter $\theta, \omega$ is of order $1/n$. The heuristic  is that a small number of bins implies a smaller number of parameters to estimate, so that the asymptotic regime is attained faster. 
Our simulations confirm this heuristic, see Section \ref{sec:simu}. 

To take a small number of bins but large enough to get identifiability, we observe in Section \ref{sse: choiceM0} a great heterogeneity among different estimators and also that some estimators have null components or cannot be computed,  when the number of bins is too small.

\section{Simulation study}
\label{sec:simu}

\subsection{ On the estimation of the risk and the selection of $M$ } 
In this section, we illustrate the results obtained in Sections \ref{se:reasons_model_selection} and \ref{se:criterion_model_selection} with simulations.
 We compare six criteria for the model selection based on $C_{CV}$ with different choices of size of training and testing sets.
We choose the regular embedded dyadic partitions, i.e. when ${\cal M}= \{2^p, p \in \mathbb{N}^* \}$ and for all $M\in {\cal M}$, $I_m=[(m-1)/M,m/M)$ for all $m < M$, $I_M=[(M-1)/M,1]$. Following Corollary \ref{co:bound_for M_n}, when $n$ is fixed, we only consider $M=2^P\leq M_n= n^3$ (i.e. $ P \leq P_n:= \lfloor 3/2 \log(n) \rfloor$).
In this part, we only consider MLE estimators with ordered components and approximated thanks to the EM algorithm.

 For $n$ fixed, the choice of the model, through $P$, is done using the criterion $C_{CV}$ based on two types of choice for $(B_b), (B_{-b})$. First, we use the framework under which we were able to prove something, i.e. Assumption \ref{hyp:size_samplin_set_ind} where all the training and testing sets are disjoints. In this context we use different sizes $a_n$ and $b_n$:
\begin{itemize}
\item $ b_n= \lceil n^{2/3}log(n)/(20)\rceil$ and $a_n=\lfloor n/(2b_n) \rfloor$ (Assumption of Proposition \ref{prop:alm_oracle_ineq}, up to $\log(n)$), leading to the criterion $C_{CV}^{D,1}$ and the choice of $P$ noted $\widehat{P}_n^{D,1} \in \argmin_{P\leq P_n} C_{CV}^{D,1}(\mathcal{I}_{2^P})$, 
\item $b_n=\lceil n^{1/3} \rceil$, $a_n=\lfloor n/(2b_n) \rfloor$,  leading to the criterion $C_{CV}^{D,2}$ and the choice of $P$ noted $\widehat{P}_n^{D,2} \in \argmin_{P\leq P_n} C_{CV}^{D,2}(\mathcal{I}_{2^P})$,
\item $a_n=\lfloor n/10 \rfloor $, $b_n = \lfloor n/(2a_n) \rfloor$, leading to  the criterion $C_{CV}^{D,3}$ and the choice of $P$ noted $\widehat{P}_n^{D,3} \in \argmin_{P\leq P_n} C_{CV}^{D,3}(\mathcal{I}_{2^P})$
\end{itemize}
We also consider the famous V-fold, where the dataset is cut into $b_n$ disjoint sets $\tilde{B}_b$ of size $a_n$, leading to training sets $B_b=\tilde{B}_b$ and testing sets $B_{-b}=\{1,\dots n\} \setminus \tilde{B}_b$. We also use different sizes $a_n$ and $b_n$:
\begin{itemize}
\item $a_n=\lfloor n^{1/3} \rfloor$, $b_n=\lfloor n/a_n \rfloor$, leading to  the criterion $C_{CV}^{V,1}$ and the choice of $P$ noted $\widehat{P}_n^{V,1} \in \argmin_{P\leq P_n} C_{CV}^{V,1}(\mathcal{I}_{2^P})$,
\item $a_n=\lfloor n^{2/3} /2\rfloor$, $b_n=\lfloor n/a_n \rfloor$, leading to the criterion $C_{CV}^{V,2}$ and the choice of $P$ noted $\widehat{P}_n^{V,2} \in \argmin_{P\leq P_n} C_{CV}^{V,2}(\mathcal{I}_{2^P})$,
\item $a_n=\lfloor n/10 \rfloor$, $b_n=\lfloor n/a_n \rfloor$, leading to  the criterion $C_{CV}^{V,3}$ and the choice of $P$ noted $\widehat{P}_n^{V,3} \in \argmin_{P\leq P_n} C_{CV}^{V,3}(\mathcal{I}_{2^P})$
.
\end{itemize}
Note that for criteria
\begin{itemize}
 \item $C_{CV}^{j,1}$,  $j\in \{D,V\}$, $a_n$ is proportional to $n^{1/3}$ up to a logarithm term,
 \item $C_{CV}^{j,2}$,  $j\in \{D,V\}$, $a_n$ is proportional to $n^{2/3}$,
 \item $C_{CV}^{j,3}$,  $j\in \{D,V\}$, $a_n$ is proportional to $n$.
\end{itemize}

We now explain how we choose $\mathcal{I}_0$. As explained earlier 
$M_0$ has to be taken small, but not too small since otherwise the model would not be identifiable. We propose to choose the smallest $M_0=2^{P_0}$ such that $M_0 \geq k+2 $ (equivalently $P_0 \geq \log(k+2)/ \log(2)$). This lower bound ensures that generically on ${\cal I}_0$ the model \eqref{eq:model2} is identifiable.

We consider three different simulation settings. In each one of them we consider the conditionally repeated sampling model, i.e.  $f_{j,1}=f_{j,2}=f_{j,3}$, both for the true distribution and for the model.
In the three cases, $k=2$ and the other parameters are given in Table \ref{ta:parameters_simulations}. So that, we work with  $P_0=2$ and $M_0=2^2=4$.

\begin{table}[ht]
\begin{tabular}{l||l|l|l|l|}
 Simu. & $k$ & $\theta^\star$ & $f^\star_{1,1}dx=f^\star_{1,2}dx=f^\star_{1,3}dx$ &  $f^\star_{2,1}dx=f^\star_{2,2}dx=f^\star_{2,3}dx$ \\
 \hline
 \hline
  $1$ &  $2$ & $(0.3,0.7)$ & $\mathcal{N}(4/5, 0.07^2)$ truncated to $[0,1]$ & $\mathcal{N}(1/3, 0.1^2)$ truncated to $[0,1]$\\
 \hline
  $2$ &  $2$  & $(0.2,0.8)$  & $\mathcal{U}((0,1))$  & $\mathcal{N}(2/3, 0.05^2) $ truncated to $[0,1]$  \\
 \hline
 $3$ &   $2$  & $(0.3,0.7)$  & $\beta(1,2)$ & $\beta(5,3)$ \\
 \hline
\end{tabular}
\caption{Values of the true parameters for simulation $1$ to $3$}
\label{ta:parameters_simulations}
\end{table}

\noindent The different emission distributions are represented in Figure \ref{fig:emission_distrib}.

\begin{figure}[ht]
\centering
    \subfigure[Simulation $1$]
        {\includegraphics[height=1.5in]{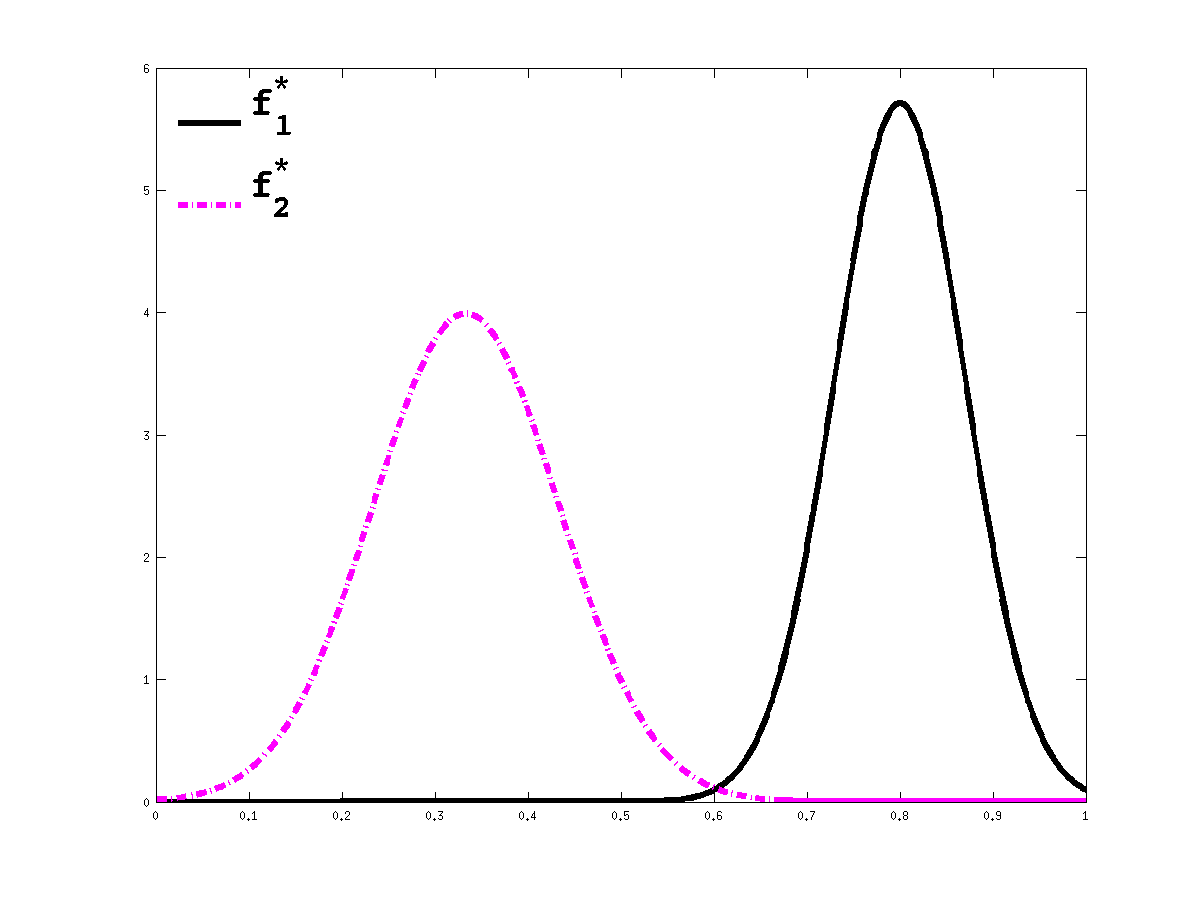}}  
    \subfigure[Simulation $2$]
    {
        \includegraphics[height=1.5in]{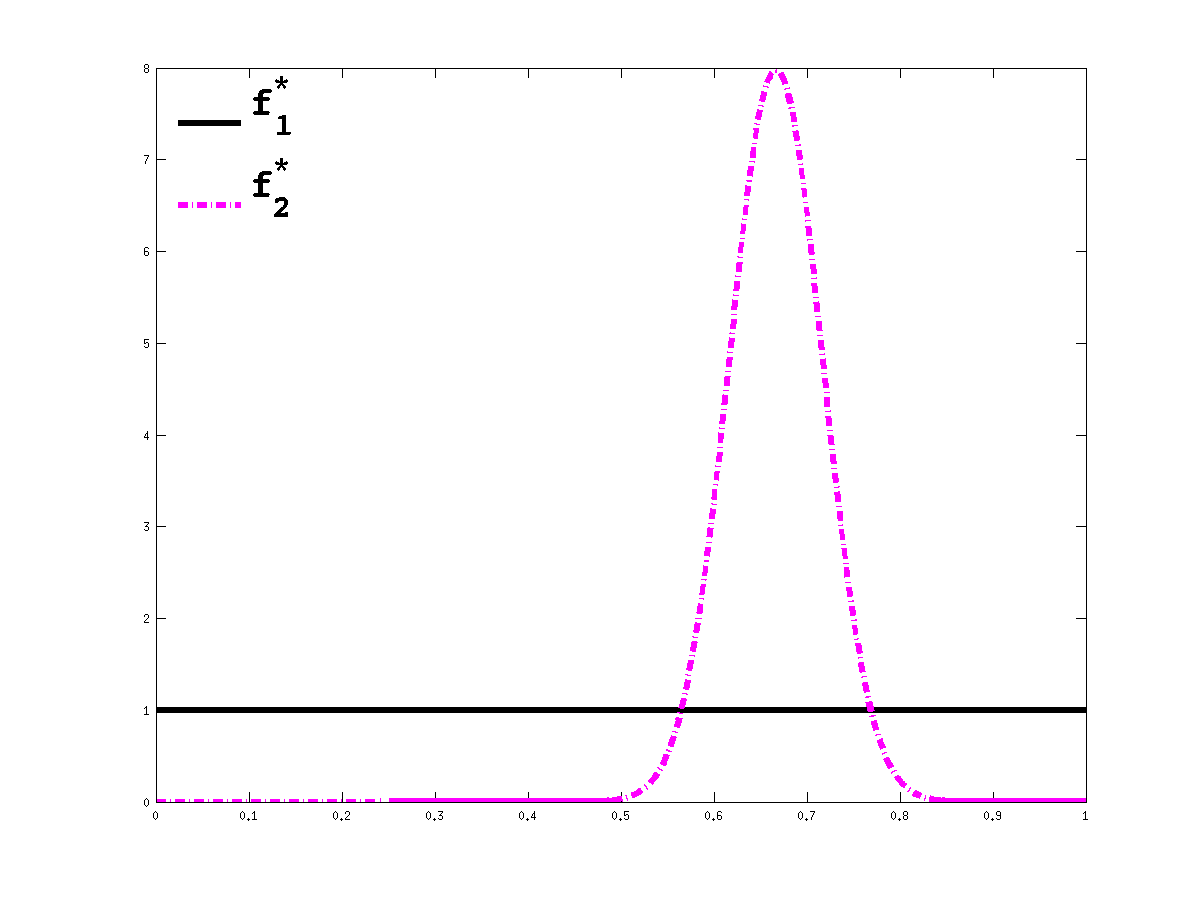}
      }
    \subfigure[Simulation $3$]
    {
        \includegraphics[height=1.5in]{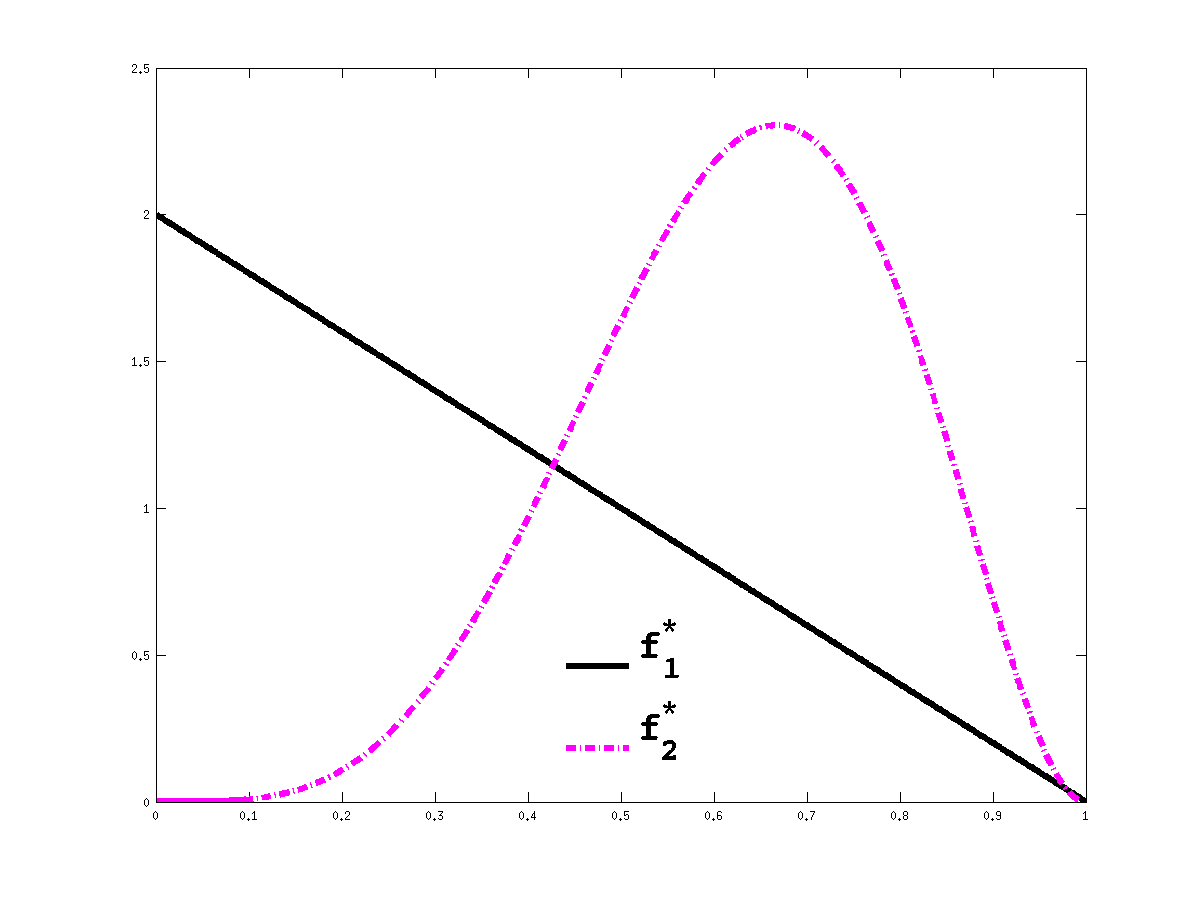}
      }
    
    \caption{Representation of the true emission distributions for simulations $1$, $2$ and $3$.}
    \label{fig:emission_distrib}
\end{figure}

In Figure \ref{fig:risk_biais_var} we display the evolution of the risk 
$R_n(\mathcal{I}_{2^P})$, the variance $Var^\star\left[\widehat{\theta}_{2^P}(X_{1:n})  \right]$  and  the squared bias 
 $\left\| \mathbb{E}^\star \left[  \widehat{\theta}_{2^P}(X_{1:n}) \right] - \theta^\star \right\|_{\mathcal{T}_k}^2$
    defined in Equation \eqref{eq:risk_bias_var} as the number of bins  $2^P$ increases, for different values of $n$ and for each of the three  true distributions. The risks, bias and variances are estimated by Monte Carlo, based on  $1000$ repeated samples and for each of them we compute the MLE using the EM algorithm. We notice that typically the bias first is either constant or slightly decreasing as $P$ increases and then increases rapidly for larger values of $P$ until it stabilizes  to the value $\|\underline{\theta}_n - \theta^\star \|_{\mathcal{T}_k}$, which is what was proved in  Proposition \ref{prop:limit_thetaM}. On the other hand  the variance is monotone non increasing as $P$ increases until $P$ becomes quite large and then it  decreases to zero (which also is a consequence of Proposition \ref{prop:limit_thetaM}) when $P$ gets large.
   As a result the risk, which is the sum of the squared bias and the variance, is typically constant or decreasing for small increasing values of $P$ and then increasing to $\| \underline{\theta}_n- {\theta^\star} \|^{2}_{\mathcal{T}_k}$ when $P$ gets large.

\begin{figure}[!ht]
\centering
    \subfigure[Simulation $1$, $n=100$]
        {\includegraphics[height=1.5in]{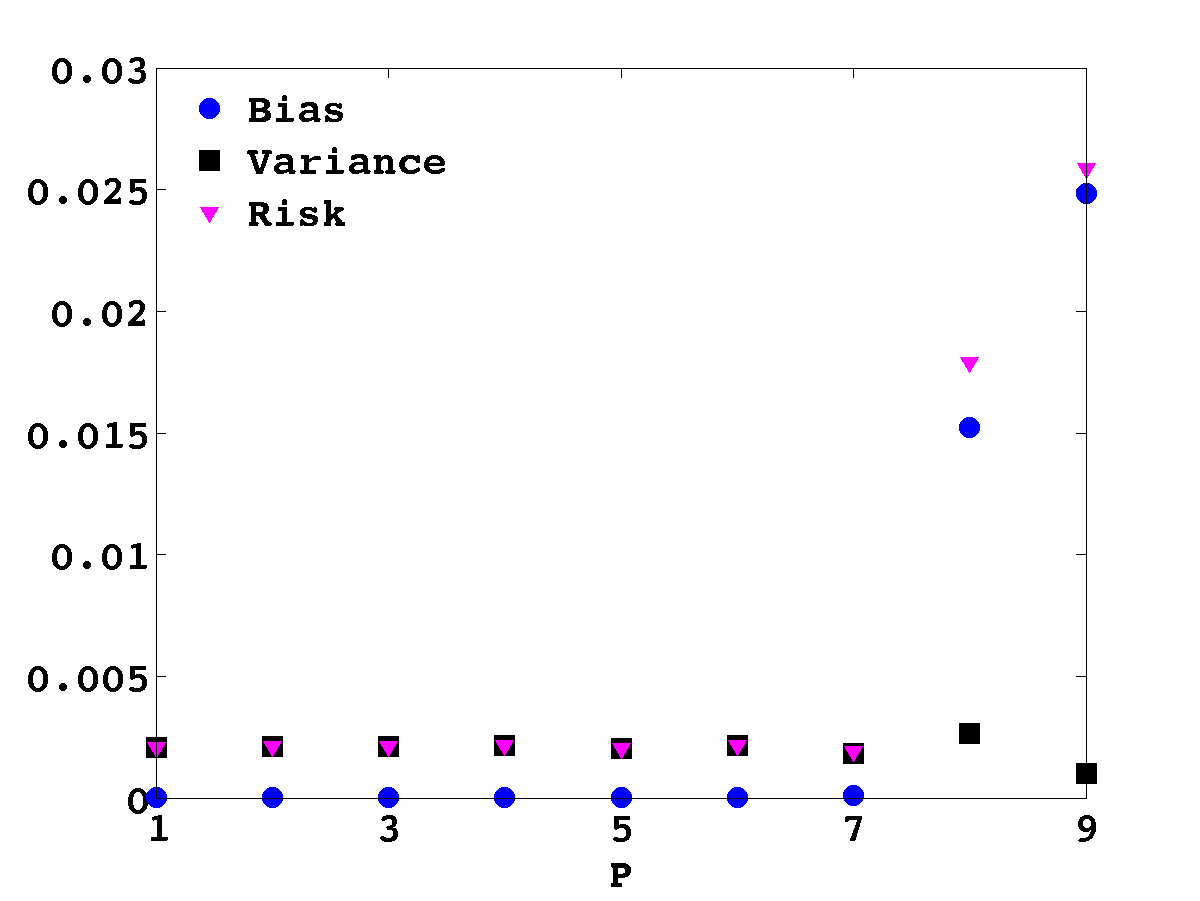}}   
         ~ 
    \subfigure[Simulation $3$, $n=50$]
    {
        \includegraphics[height=1.5in]{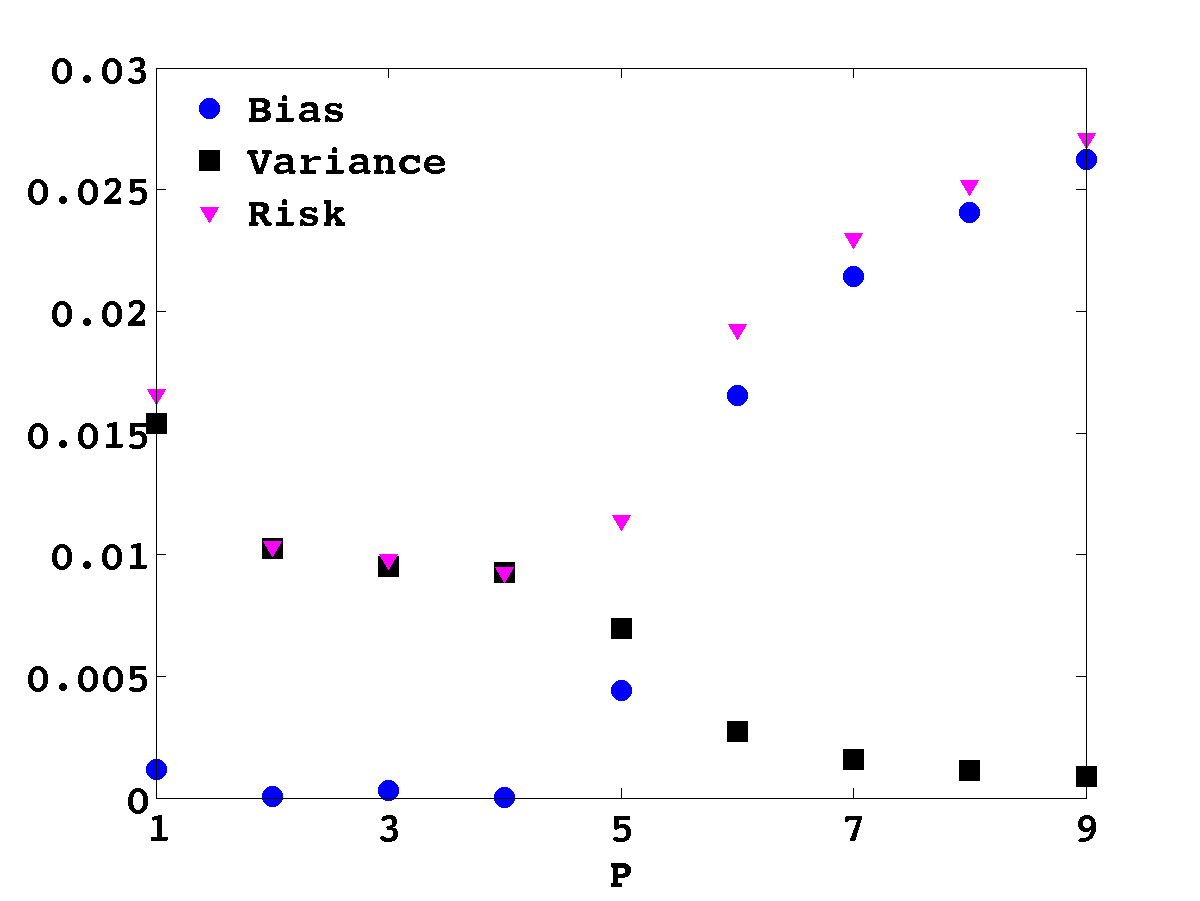}
      }
    ~ 
    \subfigure[Simulation $2$, $n=50$]
    {
        \includegraphics[height=1.5in]{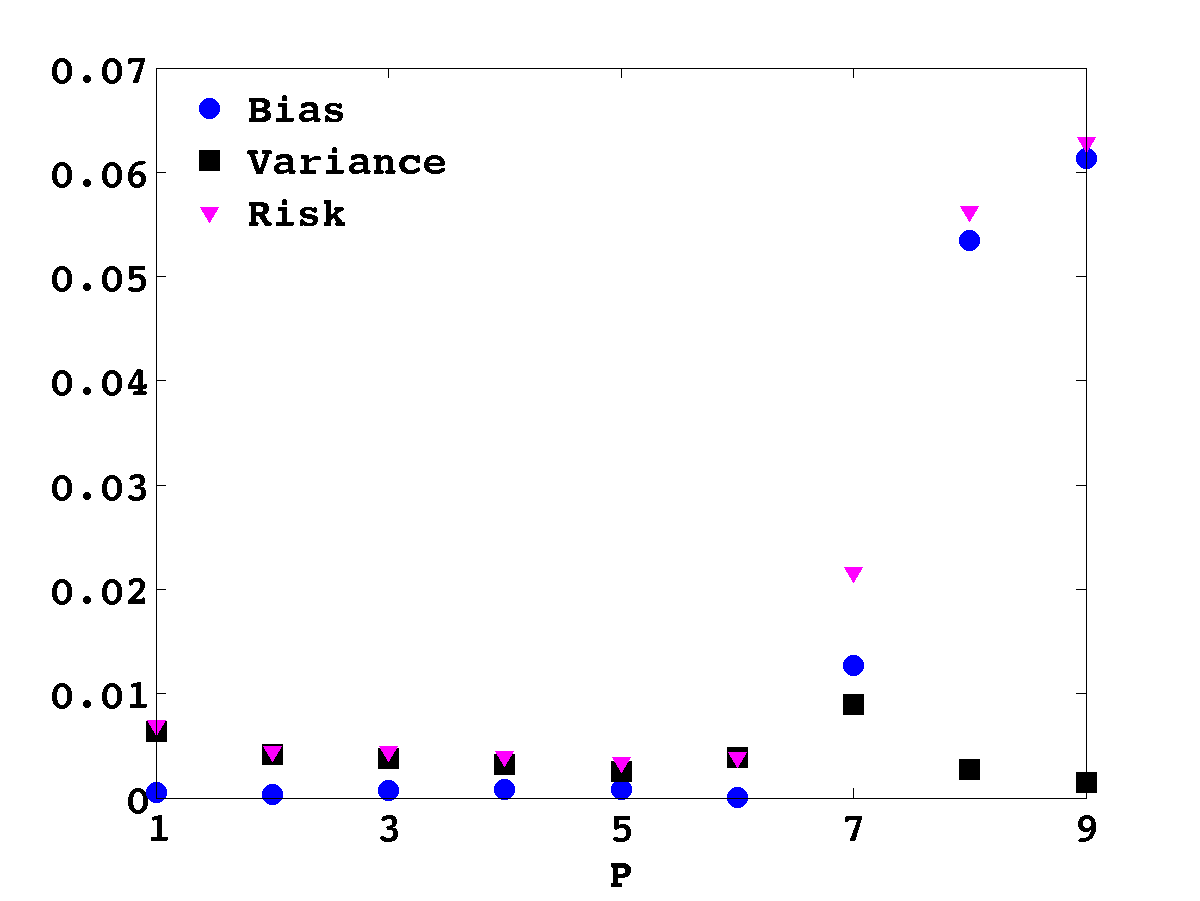}
      }
       ~ 
    \subfigure[Simulation $3$, $n=100$]
    {
        \includegraphics[height=1.5in]{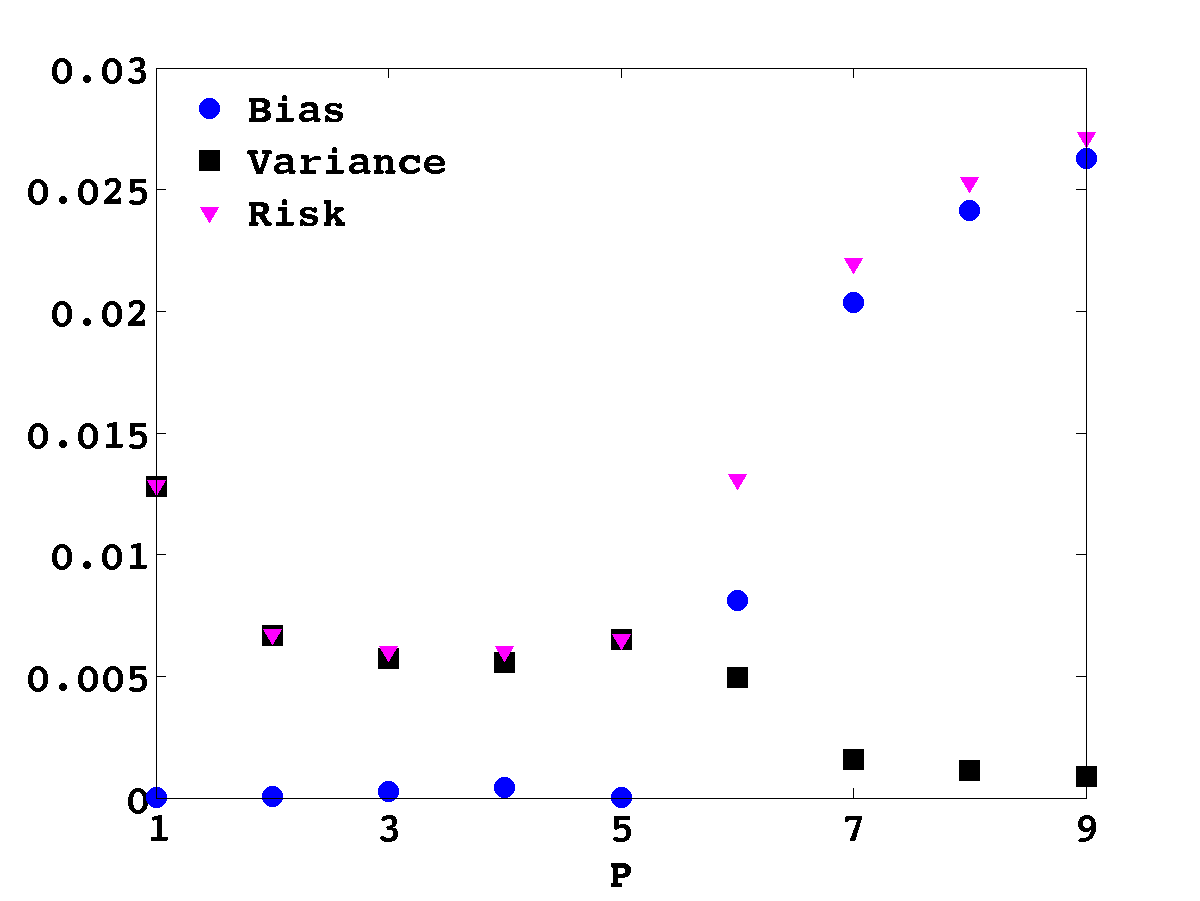}
      }
    ~ 
    \subfigure[Simulation $2$, $n=500$]
    {
        \includegraphics[height=1.5in]{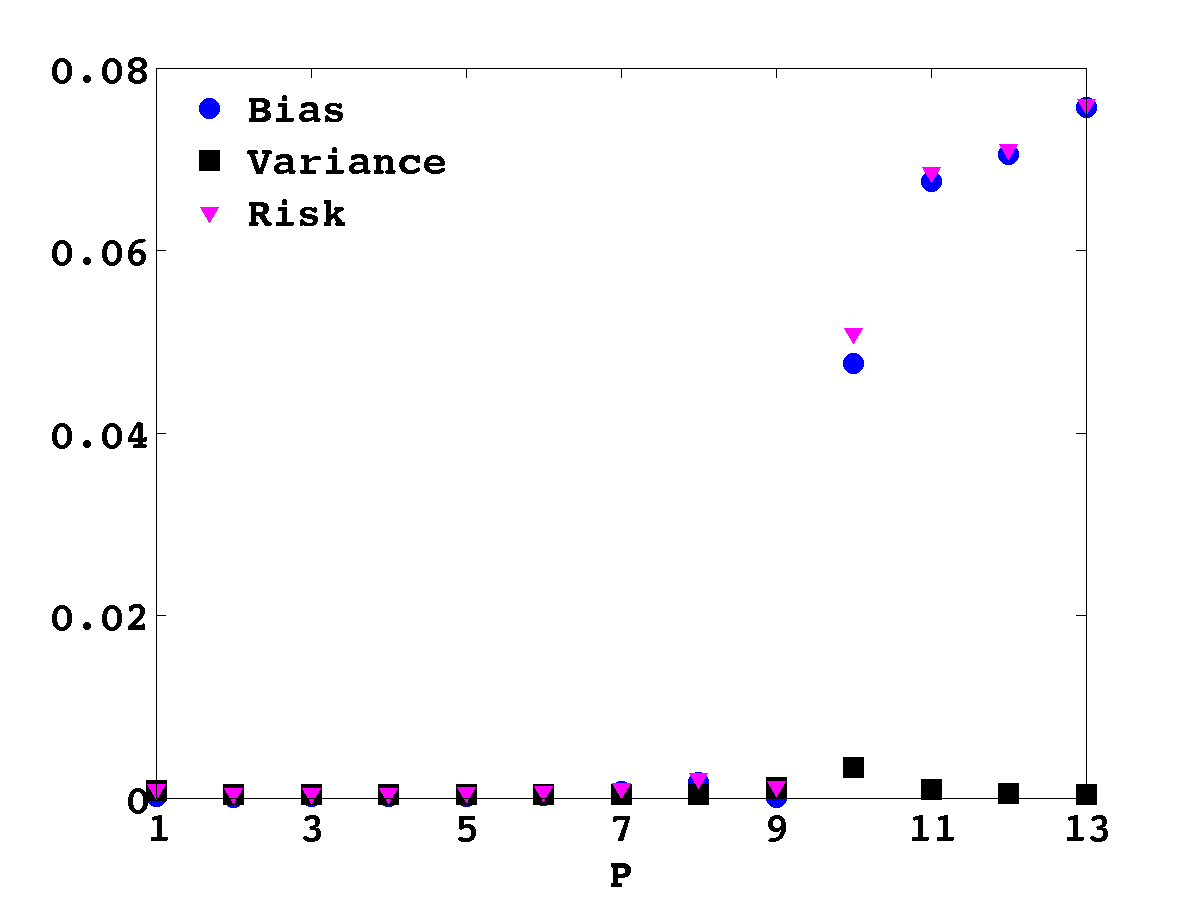}
      }
      ~ 
    \subfigure[Simulation $3$, $n=500$]
    {
        \includegraphics[height=1.5in]{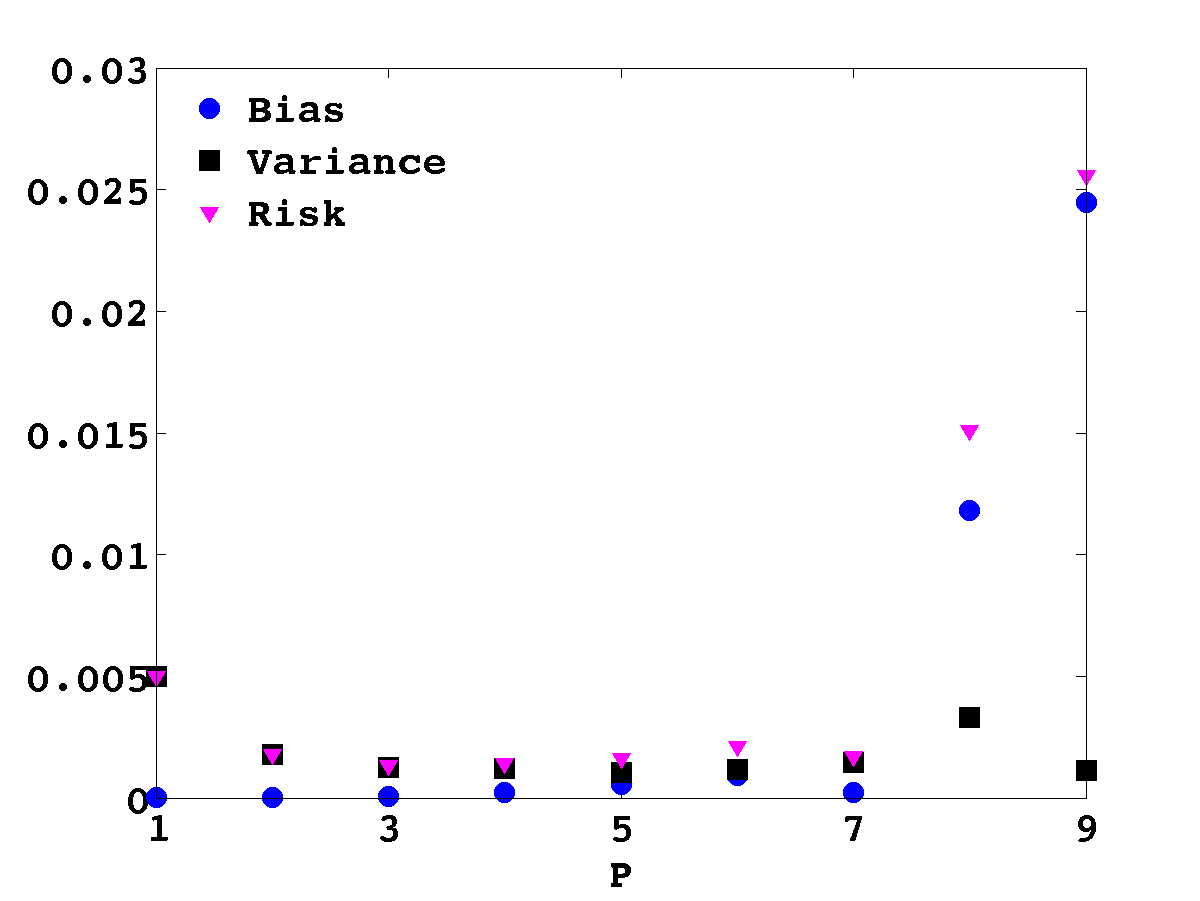}
      }
    \caption{Patterns of the risk (with black squares), the squared bias (with blue dots) and variance (with magenta triangles) with respect to $P=\log (M)/ \log (2)$ for simulations $1$, $2$ and $3$ and different values of $n$.}
    \label{fig:risk_biais_var}
\end{figure}

In real situations $R_n(\mathcal I_{2^P})$ is unknown, we now illustrate 
 the behaviour of the different criteria  $C_{CV}$ and $C_{CV1}$ and see how close to $R_n(\mathcal I_{2^P})$ they are. For the sake of conciseness we only display results for simulated data 1 and 2 and for $n=100, 500$ since they are very typical of all other simulation studies we have conducted. The results are presented in  Figure \ref{fig:pattern_criteria}, where the criteria $C_{CV}, C_{CV1}$ are computed based on a single data $X_{1:n}$. We see in figure \ref{fig:pattern_criteria} that contrarywise to $C_{CV}$, the basic cross-validated criterion  $C_{CV1}$ does not recover the behaviour of $R_n(\mathcal I_{2^P})$  correctly  as it fails to estimate the bias.  Note that we do not compare the values but the behaviour. Indeed, the criteria are used to choose the best $P$ by taking the minimum of the criterion so that the values are not important by themselves. Besides, we know that the criterion $C_{CV}$ is biased by a constant depending on $\mathcal{I}_0$. As theoretically explained in Section \ref{sec:modelselection} and as a consequence of Proposition \ref{prop:limit_thetaM}, we can see that the criteria $C_{CV1}$ are tending to $0$ when $P$ increases while it is not the case for the criteria $C_{CV}$. It is interesting to note that from  Figure \ref{fig:pattern_criteria}, the minimizer in $P$ of  $C_{CV}$ corresponds to values of the risk that are close to the minimum, we precise this impression with table \ref{ta:squared_risk}.

\begin{figure}[!ht]
\centering
    \subfigure[Simulation $1$, $n=100$, $C_{CV(1)}^{D,1}$]
        {\includegraphics[height=1.5in]{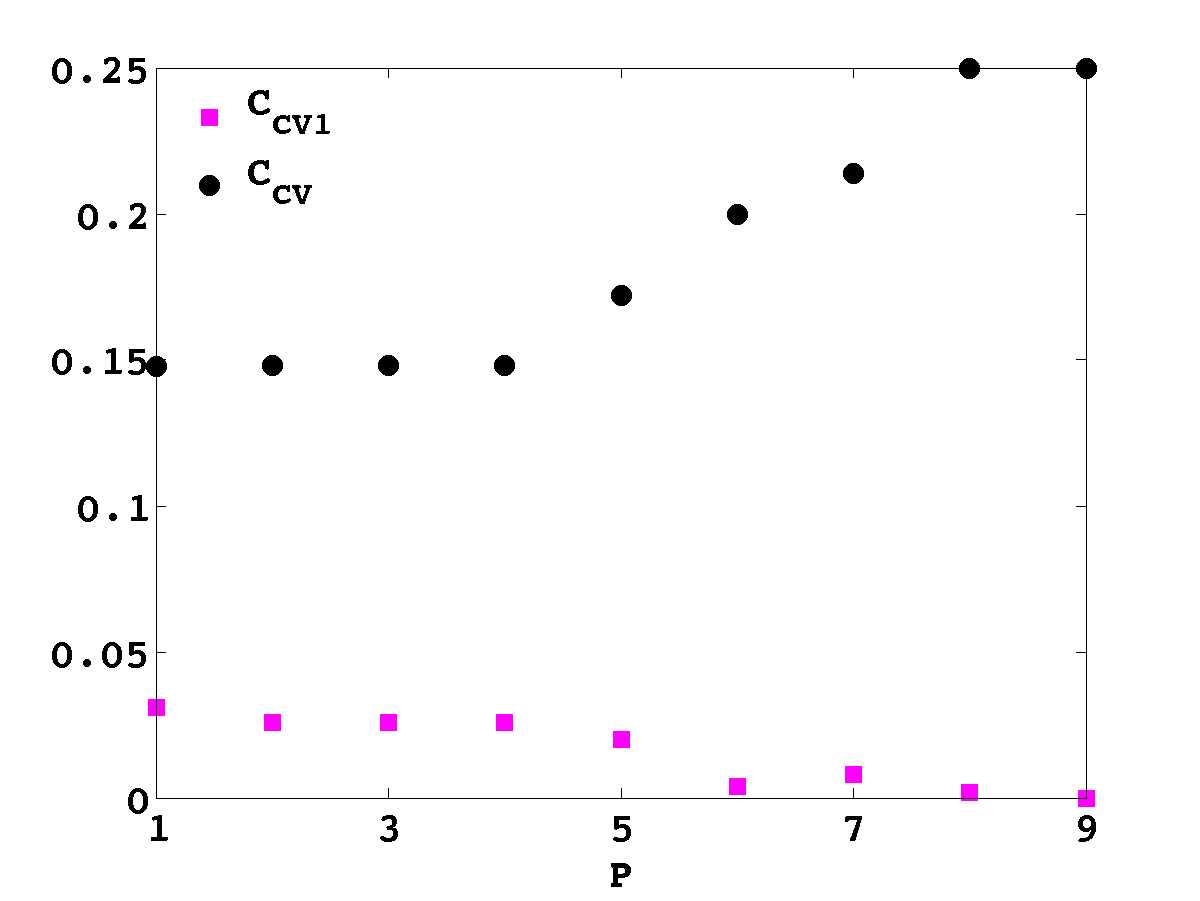}}   
    ~ 
    \subfigure[Simulation $1$, $n=100$, $C_{CV(1)}^{V,3}$]
    {
        \includegraphics[height=1.5in]{critere_ind1_J1N100.png}
      }
    ~ 
    \subfigure[Simulation $2$, $n=500$, $C_{CV(1)}^{D,1}$]
    {
        \includegraphics[height=1.5in]{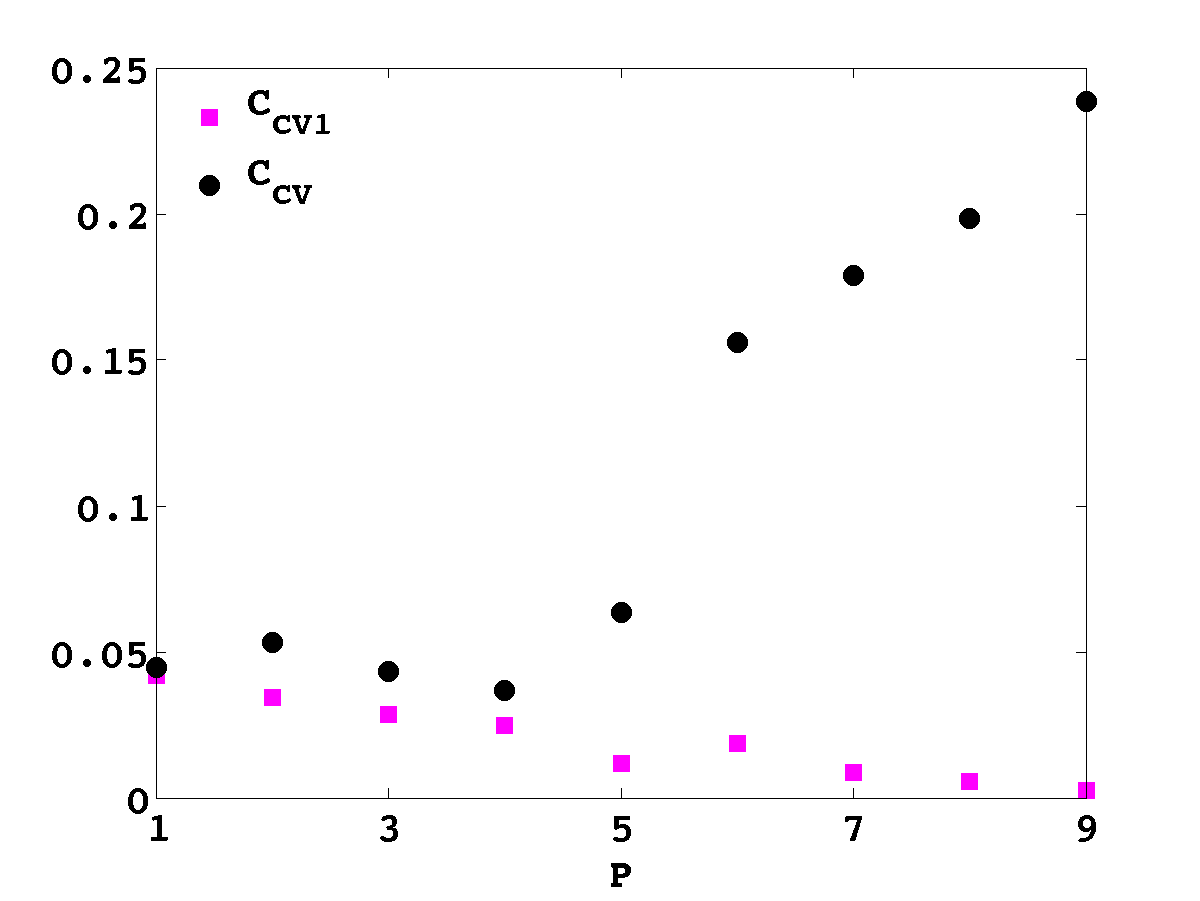}
      }
      ~ 
    \subfigure[Simulation $2$, $n=500$, $C_{CV(1)}^{V,3}$]
    {
        \includegraphics[height=1.5in]{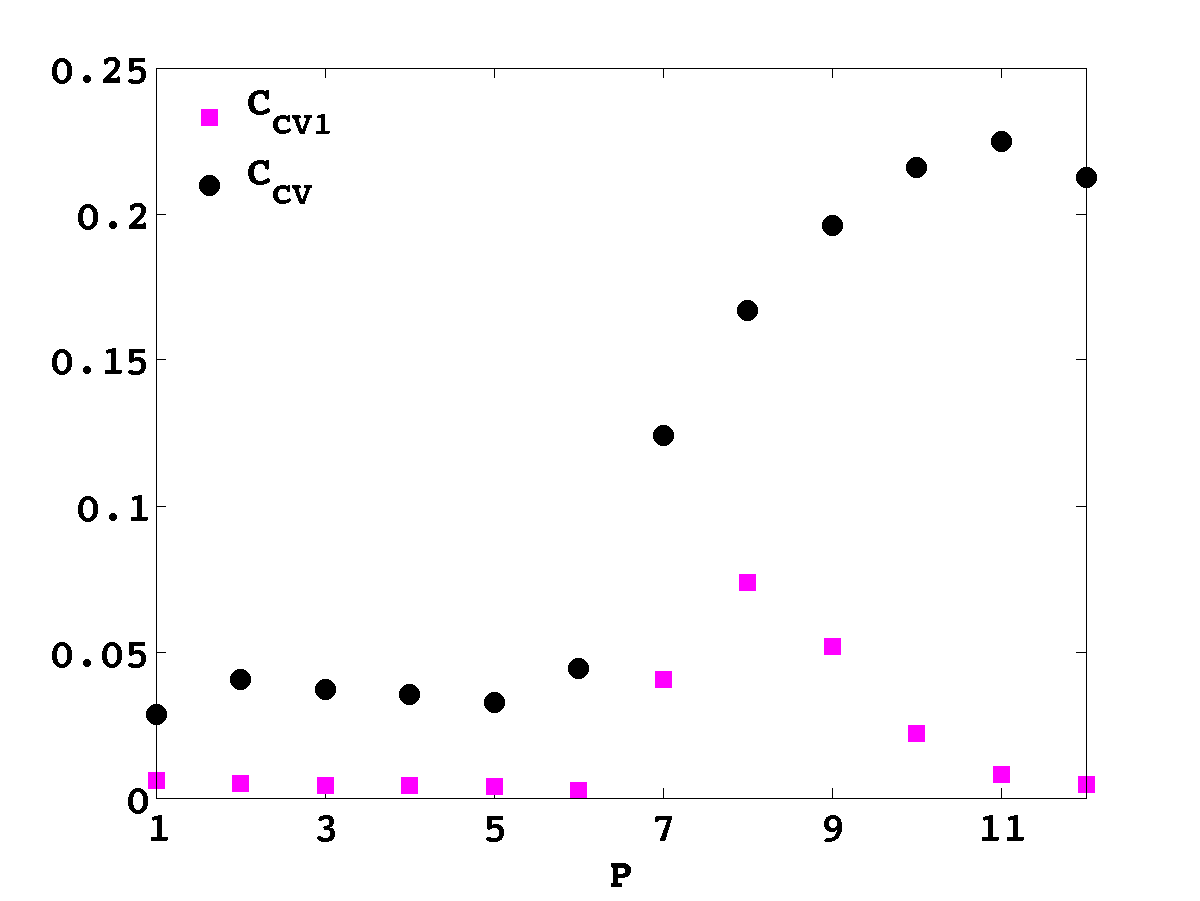}
      }
    
    \caption{behaviour of $C_{CV1}$ vs $C_{CV}$ as a function of $P$}
    \label{fig:pattern_criteria}
\end{figure}

Finally we compare the six criteria $C_{CV}^{j,c}$, $j\in \{D,V\}$, $c \in \{1,2,3\}$, by estimating the squared risk of the associated estimator 
$\widehat{\theta}_{2^{\widehat{P}_n^{j,c}}}$, 
presented
 in Table \ref{ta:squared_risk} across different sample sizes $n$ and   the three simulation set-ups (simulated data 1, 2 and 3) described above.  We can compare the six squared risk to $\sqrt{\min_{P\leq P_n} R_n(2^P)}$  and $\sqrt{ R_n(2^{P_0})}$. The different risks are estimated by Monte Carlo by repeating $100$ times the estimation. The differences of performance between the different criteria are not obvious. Besides, the performances of all the criteria are satisfactory, compared to 
 $\sqrt{\min_{P\leq P_n} R_n(2^P)}$.
  Yet, we suggest not to use criterion $C_{CV}^{V,1}$ because it is computationally more intensive than the others,  particularly when $n$ is large (because of large $b_n$). In our simulation study  $C_{CV}^{D,1}$ and $C_{CV}^{V,2}$ seem to behave slightly better than the others.

\begin{table}[ht]
\resizebox{\columnwidth}{!}{
\begin{tabular}{l|lllll|lll|lll}
 Simulation & $1$ & $1$ & $1$ &$1$ &$1$ &$2$ & $2$& $2$&$3$& $3$& $3$ \\
 $n$ & $50$ & $100$ &  $500$ & $1000$& $2000$&  $50$ & $100$ &  $500$ & $50$ & $100$ &  $500$  \\
 \hline
 $\sqrt{\min_{P\leq P_n} R_n(2^P)}$ & $0.062$ & $0.043$ & $0.020$& $0.014$ & $0.010$& $0.058$ & $0.046$ & $0.020$ & $0.096$ & $0.078$ & $0.036$ \\
 $\sqrt{ R_n(2^{P_0})}$ & $0.063$ & $0.046$ & $0.021$ & $0.015$ & $0.010$ & $0.067$ & $ 0.046$ & $0.022$ & $0.10$ & $0.082$ & $0.042$  \\
 \hline
 $\sqrt{\E^{\star} \left[\|\widehat{\theta}_{2^{\widehat{P}_n^{D,1}}}(X_{1:n})-\theta^{\star}\|^{2} \right] }$  
  & $0.069$ & $0.047$ & $0.019$ & $0.014$ & $0.011$ & $0.075$ & $0.056$ & $0.019$ & $0.12$ 	 & $0.087$ &  $0.037$ \\
 $\sqrt{\E^{\star} \left[\|\widehat{\theta}_{2^{\widehat{P}_n^{D,2}}}(X_{1:n})-\theta^{\star}\|^{2} \right] }$  
  & $0.073$ & $0.046$ & $0.022$ & $0.015$ & $0.010$ & $0.065$ & $0.056$ & $0.025$ & $0.10$ & $0.087$ &  $0.046$\\
 $\sqrt{\E^{\star} \left[\|\widehat{\theta}_{2^{\widehat{P}_n^{D,3}}}(X_{1:n})-\theta^{\star}\|^{2} \right] }$  &
 $0.086$ & $0.047$ & $0.021$ & $0.014$ & $0.010$ & $0.087$ & $0.056$ & $0.026$ & $0.11$ & $0.087$ & $0.041$ \\
 $\sqrt{\E^{\star} \left[\|\widehat{\theta}_{2^{\widehat{P}_n^{V,1}}}(X_{1:n})-\theta^{\star}\|^{2} \right] }$  
  & $0.091$ & $0.046$ & $0.021$ & $0.013$ & $0.009$ & $0.104$ & $0.055$ & $0.022$ & $0.11$ & $0.087$ & $0.053$ \\
$\sqrt{\E^{\star} \left[\|\widehat{\theta}_{2^{\widehat{P}_n^{V,2}}}(X_{1:n})-\theta^{\star}\|^{2} \right] }$  
 & $0.069$ & $0.046$ & $ 0.019$ & $0.013$ & $ 0.010$ & $0.070$ & $0.049$ & $0.022$ & $0.12$ & $0.084$ & $0.036$ \\
 $\sqrt{\E^{\star} \left[\|\widehat{\theta}_{2^{\widehat{P}_n^{V,3}}}(X_{1:n})-\theta^{\star}\|^{2} \right] }$  
  & $0.103$ & $0.046$ & $0.019$ & $0.014$ & $ 0.009$ & $0.10$ & $0.049$ & $0.022$ & $0.14$ & $0.083$ & $0.035$ \\
 \hline
\end{tabular}}
\caption{Comparison of the squared risk of estimators associated to different criteria}\label{ta:squared_risk}
\end{table}

These results confirm that by using $M_0$ small, the criterion behaves correctly. Moreover, the fact that the choice of $\widehat{I}_n$ corresponds to a risk associated with $a_n<n$ observations does not seem to be a conservative choice even in a finite horizon (i.e. when $n$ is fixed). We were expecting this behaviour asymptotically but not necessarily in a finite horizon.

\subsection{ On the choice of $M_0$ }\label{sse: choiceM0}

To compute the different criteria $C_{CV}^{j,i}$, $j\in\{D,V\}$, $i\in\{1,2,3\}$, we proposed to choose  $M_0$ as small as possible but for which the model is identifiable up to label switching. Given $\min_{1 \leq j \leq k} \theta^\star_j >0$, the model associated to the parameter space $ $ is identifiable as soon as the $k$ vectors $(\omega^\star_{1,c,\cdot})_{M_0}$, $\dots$, $(\omega^\star_{k,c,\cdot})_{M_0}$ in $\Delta_{M_0}$ are linearly independent for all $c\in\{1,2,3\}$.
 Considering the dimension of linear spaces, we may choose $M_0=k+1$ or $k+2$. Then we should generically avoid issues with identifiability. We chose such a $M_0$ in the previous simulations. We now study the impact of a choice of $M_0$ that would be too small .

To do so we have  simulated data from $f^\star_{2,1}(y)=f^\star_{2,2}(y)=f^\star_{2,3}(y)=1 $ and $f^\star_{1,1}(y)=f^\star_{1,2}(y)=f^\star_{1,3}(y)=1+\cos(2\pi (y+\epsilon)) $, for $\epsilon=0.25$, $0.3$, $0.4$ and $0.5$ and $n=500$, 1000 and 2000. In this case the smallest possible value for $M_0$ based on the regular grid on $[0,1]$ is $M_0= 2$ for $\epsilon \neq 0.5$ whereas if $\epsilon = 0.5$ the model is non identifiable at $M_0=2$ but becomes identifiable when $M_0 = 4$. For each of these simulation data we have computed various estimators: MLEs based on the EM algorithm initiated at different values, the posterior mean and the MAP estimator computed from a Gibbs sample algorithm with a  Dirichlet   prior distribution on $\theta$ and independently a  Dirichlet   prior distribution on each $\omega_{j,1} $, for $1\leq j \leq k$. We noticed that 
the EM algorithm with different initializations were very heterogeneous. Moreover, the MAP estimator, posterior mean and spectral estimators often had one of the $\hat \theta_j$ null or close to 0. Sometimes, the spectral estimator could not be computed.

The explanation for such behaviour is that when the model is not identifiable, one $\theta^{\star}_j$ may be null
 or the vectors $(\omega^\star_{1,c,\cdot})_{M_0}$, $\dots$ ,$(\omega^\star_{k,c,\cdot})_{M_0}$ may be linearly dependent for some $c\in\{1,2,3\}$. In this case, the likelihood will have multiple modes  (apart from those arising because of label switching). Hence a way to check that $M_0$ is not too small is to compute multiple initialisation of the EM algorithm if the MLE is estimated or to look for very small values of $\hat \theta_j$ in the case of Bayesian estimators (possibly also running multiple MCMC chains with different initial values). In practice we suggest that this analysis be conducted for a few number of values $M_0$ and then to select the value that leads to the most stable results.

 To illustrate this, we present a simulation study  where the number of estimators is $S=14$. The 10 first estimators were obtained using the EM algorithm with different random initializations, we also considered the spectral estimator proposed in \cite{kakade14} and an estimator obtained with the EM algorithm with the spectral estimator as initialization.
The last two estimators were the MAP estimator and the posterior mean. We considered regular partitions with $M_0=2$, $4$ and $6$ bins.  To present the results in a concise way we have summarized them in Figure \ref{fig:choix_M_0}, using the indicator
\[
ID(M_0)=
\frac{S ~ \mathds{1}_{\forall s\leq S, ~ j\leq 2,~ \hat{\theta}^s_{M_0,j} > 1/\sqrt{n} }}{\sum_{s=1}^S \left( \min_{j= 1,2} \hat{\theta}^s_{M_0,j} - \frac{1}{S}\sum_{t=1}^S  \min_{j= 1,2}\hat{\theta}^t_{M_0,j}  \right)^2}
\]
with $k=2$ and $\theta^\star =(0.7,0.3)$. 
Thus when $ID(M_0)=0$, there is a suspicion that the associated model is not identifiable and another partition should be chosen. It appears that for all $n$, when $\epsilon = 0.5$ $M_0=2$ always appears as having a pathological behaviour and for the other values of $\epsilon$ this value is accepted for large values of $n$. 

\begin{figure}[!ht]
\centering
    \subfigure[ $n=500$]
	{
	\begin{tabular}{l|cccc}
	 & $\epsilon=0.25$ &  $\epsilon=0.3$ & $\epsilon=0.4$ & $\epsilon=0.5$ \\
	 \hline
	 $M_0=2$ & 0  & 0 & 0 & 0\\
	  $M_0=4$ & {\bf 24} & 0 & {\bf 58} & 0 \\
	   $M_0=6$ & 24 & {\bf 10 }  & 0 & {\bf 28} \\
	\end{tabular}        
	        }   
	    ~ 
    \subfigure[$n=1000$]
    {
       \begin{tabular}{l|cccc}
 & $\epsilon=0.25$ &  $\epsilon=0.3$ & $\epsilon=0.4$ & $\epsilon=0.5$ \\
 \hline
 $M_0=2$ & 0 & 0 & 0 & 0\\
  $M_0=4$ & {\bf 26 } & {\bf 28 }  & {\bf 25 } & {\bf 63} \\
   $M_0=6$ & 0 & 26 & 16 & 10\\
\end{tabular}
      }
    ~ 
    \subfigure[ $n=2000$]
    {
        \begin{tabular}{l|cccc}
 & $\epsilon=0.25$ &  $\epsilon=0.3$ & $\epsilon=0.4$ & $\epsilon=0.5$ \\
 \hline
 $M_0=2$ & {\bf 113 } & {\bf 116} & 0 & 0\\
  $M_0=4$ & 31 & 38 & {\bf 38 } & {\bf 47} \\
   $M_0=6$ & 18 & 40 & 14 & 19 \\
\end{tabular}
      }

    \caption{The identifiability criterion $ID(M_0)$ for different $M_0$, $n$ and $\epsilon$.}
    \label{fig:choix_M_0}
\end{figure}

\section{Conclusion and discussion}
\label{sec:discussion}

To sum up our results, we propose semiparametric estimators of the mixing weights in a mixture with a finite number of components and unspecified emission distributions. These estimators are constructed using an approximate model for the mixture where the emission  densities are modelled as piecewise constant functions on fixed partitions of the sampling space. This approximate model is thus parametric and regular and more importantly well specified as far as the weight parameters $\theta$ are concerned. From Theorem  \ref{th:FreeLunch} we have that for all $M \geq M_0$ ,    $\sqrt{n}\left(\widehat{\theta}_{\tilde{M},n}-\theta^\star \right)\leadsto \mathcal{N}\left(\theta^\star, \tilde{J}_{\tilde{M}}^{-1} \right)$ as $n$ goes to infinity and that $
\tilde{J}_{\tilde{M}}^{-1} \rightarrow \tilde{J}^{-1}$ as $M$ goes to infinity (and similarly from a Bayesian point of view). Moreover we have proved in Section \ref{se:reasons_model_selection} that for all $n$, as $M$  goes to infinity,  $\widehat{\theta}_{M,n} \rightarrow \bar \theta_n $ and that  as $n \rightarrow +\infty $, $ \bar \theta_n \rightarrow (1/k, \cdots, 1/k)$ whatever the true value $\theta^*$ of the parameter.   These two results show that we can find a sequence $M_n$ going to infinity such that $\sqrt{n}\left(\widehat{\theta}_{\tilde{M_n},n}-\theta^\star \right)\leadsto \mathcal{N}\left(\theta^\star, \tilde{J}^{-1} \right)$ but also that we cannot choose $M_n$ going to infinity arbitrarily fast. It is thus important to determine a procedure to select $M$, for finite $n$.

To choose $M_n$ in practice, for finite $n$, we  propose in Section \ref{se:criterion_model_selection} an approach which consists in minimizing an estimate of the quadratic risk $R_n(\mathcal I)$ in the partition $\mathcal I$, as a way to ensure that the asymptotic variance of  $\sqrt{n}(\widehat \theta - \theta)  $ is close to $\tilde J^{-1}$ and that the quantity $\sqrt{n}(\widehat \theta - \theta) $ is asymptotically stable. The construction of an estimator of $R_n(\mathcal I)$ is not trivial due to the strong non linearity of the maximum likelihood estimator in mixture models and we use a reference model with a small number of bins $M_0$ as a proxy for an unbiased estimator $\theta$, together with a cross validation approach to approximate $R_{a_n}(\mathcal I$ for all partition $\mathcal I$ with $a_n = o(n)$. This leads at best to a minimization of the risk $R_{a_n}(\mathcal I_M)$ instead of $R_n(\mathcal I_M)$,  however this is it not per se problematic since a major concern is to ensure that  $M_n$ is not too large.

In the construction of our estimation procedure (either by MLE or based on the posterior distribution) we have considered the same partitionning of $[0,1]$ for each coordinate $c\in \{1, 2, 3\}$. This can be relaxed easily by using different partitions accross coordinates, if one wishes to do so to adapt to different smoothness of emission densities for instance. However, this would require choices of $M$ for each coordinate. We believe that our theoretical results would stay true.
We did more simulations in this setting and we observed that, when the emission distributions are distinct in each direction, choosing different $M$ for each coordinates is time consuming and does not really improve the estimations of $\theta$, at least in our examples.

 We have also presented our results under some seemingly restrictive assumptions, which we now discuss.

\subsection{On the structural assumptions on the model } 

In model \eqref{eq:model1}, it is assumed that each individual has three conditionally independent observations in $[0,1]$ each. Obviously this assumptions can be relaxed to any number $p$ of conditionally independent observations with $p \geq 3$ without modifying the conclusions of our results.

Also, the method of estimation relies heavily on the fact that the $X_{i,c}$'s belong to $[0,1]$. This is not such a restrictive assumption since one can transform any random variable on $\mathbb R$ into $[0,1]$, writing $X_{i,c} = G_c(\tilde{X}_{i,c})$ , where $G_c$ is a  given cumulative distribution on $\mathbb R$ and $\tilde X_{i,c}$ is the original observation. Then the conditional densities are obtained as 
$$ f_{j,c}^\star(x_c) =  f_{\tilde X;j,c}(G_c^{-1}(x_c))/g_c(G_c^{-1}(x_c)), \quad j \leq k, c\in \{1, \cdots, 3\}$$
and Assumption \eqref{moment:cond} becomes that for all $c\in \{1, \cdots, 3\}$,
\begin{equation} \label{moment2}
0 < \text{liminf}_{x_c} \frac{ f_{\tilde X;j_1,c}(G_c^{-1}(x_c)) }{ f_{\tilde X;j_2,c}(G_c^{-1}(x_c))} \leq \text{limsup}_{x_c} \frac{ f_{\tilde X;j_1,c}(G_c^{-1}(x_c)) } { f_{\tilde X;j_2,c}(G_c^{-1}(x_c))}< +\infty, 
\end{equation}
which means that the densities of the observations within each group have all the same tail behaviour. Note that a common assumption found in the literature for estimation of densities on $[0,1]$ is that the densities are bounded from above and below, which in the above framework of transformations $G_c$ amounts to saying that $f_{\tilde X;j_1,c}$'s have all the same tail behaviours as $g(\cdot )$. This is a much stronger assumption because it would mean that the tail behaviour of the densities $f_{\tilde X;j_1,c}$ is known a priori, whereas \eqref{moment2} only means that the tails are the same between the components of the mixtures but they not need to be the same as those of $g$. 

Finally we have considered univariate conditional observations $X_{i,c} \in \mathbb R$, again this can be relaxed easily by considering partitions of $[0,1]^d$ with $d \geq 1$ if $X_{i,c} \in [0,1]^d$. In this case the first part of Assumption (A2) needs to be replaced by :\\
$\bullet$ There exists $a>0$ such that for all $M$, for all $I_{m}$ in ${\cal I}_{M}$, there exists an open ball $I$ such that $I_{m}\subset I$ and $|I_{m}| \geq a |I|$.

\subsection{Extensions to Hidden Markov models } 

Finite mixture models all have the property that, when the approximation space for the emission distributions is that of step functions (histograms), then the model stays true for the observation process, but associated to the summary of the observations made of the counts in each bins.  This leads to a proper and well specified likelihood for the parameter $\theta, w$ and there is no problem of model misspecification as fgar as $\theta$ is concerned even when the number of bins is fixed and small.  We expect the results obtained in this paper to remain valid 
for nonparametric hidden Markov models with translated emission distributions studied in \cite{gassiat:rousseau:13} or for general nonparametric finite state space hidden Markov models studied in \cite{NPHMMYCE}, \cite{Ver13} and \cite{YCES}.
In the latter, the parameter describing the probability distribution of the latent variable is the transition matrix of the hidden Markov chain. However, semiparametric asymptotic theory for dependent observations is much more involved, see  \cite{MR1814795} for the ground principles. It seems difficult to identify the score functions and the efficient Fisher information matrices for hidden Markov models even in the parametric approximation model, so that to get results such as Theorem \ref{th:FreeLunch} could be quite challenging, nevertheless we think that the results obtained here pave the way to obtaining semi-parametric efficient estimation of the transition matrix in nonparametric hidden Markov models. 

\section{Proofs}
\label{sec:proofs}  
  
\subsection{Proof of Proposition \ref{prop:fisher}}
\label{subsec:proof:fisher}
Let us first prove that for large enough $M$,   the measures $f^{\star}_{1,c;M}dx,\ldots,f^{\star}_{k,c;M}dx$ are linearly independent. Indeed, if it is not the case, there exists a subsequence $M_{p}$ tending to infinity as $p$ tends to infinity and a sequence $(\alpha^{(p)})_{p\geq 1}$ in the unit ball of $\R^{k}$ such that for all $p\geq 1$,
$$
\sum_{j=1}^{k}\alpha^{(p)}_{j}f^{\star}_{j,c;M_{p}}(x)=0
$$
Lebesgue a.e. Let $\alpha=(\alpha_{1},\ldots,\alpha_{k})$ be a limit point of $(\alpha^{(p)})_{p\geq 1}$ in the unit ball of $\R^{k}$. Using Assumption (A.2) and Corollary 1.7 in Chapter 3 of 
\cite{MR2129625}, we have that as $p$ tends to infinity, $f^{\star}_{j,c;M_{p}}(x)$ converges to $f^{\star}_{j,c}(x)$ Lebesgue a.e. so that we obtain 
$\sum_{j=1}^{k}\alpha_{j}f^{\star}_{j,c}(x)=0
$
Lebesgue a.e., contradicting Assumption (A1).\\
Fix now $M$ large enough so that  the measures $f^{\star}_{1,c;M}dx,\ldots,f^{\star}_{k,c;M}dx$ are linearly independent. Then, one may use the spectral method described in \cite{kakade14} to get estimators $\widehat{\theta}_{sp}$ and  $\widehat{\boldsymbol{\omega}}_{M;sp}$ of the parameters $\theta$ and $\boldsymbol{\omega}_{M}$ from a sample of the multinomial distribution associated to density $g_{\theta,\boldsymbol{\omega};M}$. The estimator uses eigenvalues and eigenvectors computed from the empirical estimator of the multinomial distribution. But in a neighborhood of $\theta^{\star}$ and $\boldsymbol{\omega}^{\star}$, this is a continuously derivative procedure, and since on this neighborhood, classical deviation probabilities on empirical means hold uniformly, we get easily that for any vector $V\in \R^{k}$, there exists $K>0$ such that for all $c>0$,  for large enough $n$ (the size of the sample):
$$
\sup_{\|\theta-\theta^{\star}\| \leq \frac{c}{\sqrt{n}}} 
{\E_{\theta}\left[\left(\sqrt{n} \langle \widehat{\theta}_{sp}-\theta, V \rangle \right)^{2}\right]} \leq K.
$$
Now, the multinomial model is differentiable in quadratic mean, and following the proof of Theorem 4 in \cite{EGGS} one gets that, if $V^{T} \tilde{J}_{M} V = 0$, then
$$
\lim_{c\rightarrow +\infty} \lim_{n\rightarrow +\infty} \sup_{\|\theta-\theta^{\star}\| \leq \frac{c}{\sqrt{n}}} 
{\E_{\theta}\left[\left(\sqrt{n}\langle \widehat{\theta}_{sp}-\theta, V \rangle \right)^{2}\right]}=+\infty.
$$
 Thus  for all $V \in \R^{k}$,  $V^{T} \tilde{J}_{M} V \neq 0$, so that $\tilde{J}_{M}$ is not singular.

\subsection{Proof of Proposition \ref{prop:FisherGrow}}
\label{subsec:proof:FisherGrow}
We prove the proposition when $M_{1}=M$, $M_{2}=M+1$, ${\cal I}_{M}=\{I_{1},\ldots,I_{M}\}$ and ${\cal I}_{M+1}=\{I_{1},\ldots,I_{M,0},I_{M,1}\}$ with $I_{M}=I_{M,0}\cup I_{M,1}$, which is sufficient by induction. We denote $(\omega^{(M)}_{j,c,m})_{j,c, m}$ the parameter $\boldsymbol{\omega}$ in the model with partition ${\cal I}_{M}$ and $(\omega^{(M+1)}_{j,c,m})_{j,c, m }$ the parameter $\boldsymbol{\omega}$ in the model with partition ${\cal I}_{M+1}$.
Define $b\in (0,1)$, $\alpha_{j,c}\in (0,1)$, $j=1,\ldots,k$, $c=1,2,3$ so that
\begin{multline*}
|I_{M,0}|=(1-b) |I_{M}|,\;|I_{M,1}|=b |I_{M}|,\;\omega^{(M+1)}_{j,c,M}=(1-\alpha_{j,c})\omega^{(M)}_{j,c,M},\; \\ \omega^{(M+1)}_{j,c,M+1} = \alpha_{j,c}\omega^{(M)}_{j,c,M}.
\end{multline*}
Then, we may write 
$$
g_{\theta,\boldsymbol{\omega};M}(\mathbf{x})=\sum_{j=1}^{k}\theta_{j}\prod_{c=1}^{3}\prod_{m=1}^{M}\left(\frac{\omega^{(M)}_{j,c,m}}{|I_{m}|}\right)^{\1_{I_{m}}(x_{c})}
$$
and
\begin{align*}
&g_{\theta,\boldsymbol{\omega};M+1}(\mathbf{x})
\\ &=\sum_{j=1}^{k}\theta_{j}\prod_{c=1}^{3}\prod_{m=1}^{M-1}\left(\frac{\omega^{(M+1)}_{j,c,m}}{|I_{m}|}\right)^{\1_{I_{m}}(x_{c})}
\left[ \left(\frac{\omega^{(M+1)}_{j,c,M}}{|I_{M,0}|}\right)^{\1_{I_{M,0}}(x_{c})}\left(\frac{\omega^{(M+1)}_{j,c,M+1}}{|I_{M,1}|}\right)^{\1_{I_{M,1}}(x_{c})}\right]
\\
&=
\sum_{j=1}^{k}\theta_{j}\prod_{c=1}^{3}\prod_{m=1}^{M}\left(\frac{\omega^{(M)}_{j,c,m}}{|I_{m}|}\right)^{\1_{I_{m}}(x_{c})}
\left[\left(\frac{\alpha_{j,c}}{b}\right)^{\1_{I_{M,1}}(x_{c})}\left(\frac{1-\alpha_{j,c}}{1-b}\right)^{\1_{I_{M,0}}(x_{c})}\right].
\end{align*}
Thus, when $x_{c}\notin I_{M}$ for $c=1,2,3$, $g_{\theta,\boldsymbol{\omega};M+1}(\mathbf{x})=g_{\theta,\boldsymbol{\omega};M}(\mathbf{x})$ and computations have to take care of $\mathbf{x}$'s such that for some $c$, $x_{c}\in I_{M}$. If we parametrize  the model with partition ${\cal I}_{M+1}$ using the parameter $\left(\theta,(\omega^{(M)}_{j,c,m}),(\alpha_{j,c})\right)$ we get the same efficient Fisher information for $\theta$ as when parametrizing with $\left(\theta,(\omega^{(M+1)}_{j,c,m})\right)$.  Define the function $D$ as the difference between the gradient of $\log g_{\theta,\boldsymbol{\omega};M+1}$ and that of $\log g_{\theta,\boldsymbol{\omega};M}(\mathbf{x})$ with respect to the parameter $\left(\theta,(\omega^{(M)}_{j,c,m}),(\alpha_{j,c})\right)$:
 $$
 D(\mathbf{x}):=\nabla \log g_{\theta,\boldsymbol{\omega};M+1}(\mathbf{x}) - \nabla \log g_{\theta,\boldsymbol{\omega};M}(\mathbf{x}),
 $$
in particular the last coordinates of  $\nabla \log g_{\theta,\boldsymbol{\omega};M}(\mathbf{x})$ corresponding to the derivatives with respect to $(\alpha_{j,c})$ are zero.
Let us denote $K^{(M+1)}$ the Fisher information obtained for this new parametrization, that is 
\[K^{(M+1)}=\E^{\star}[(\nabla \log g_{\theta,\boldsymbol{\omega};M+1}(X))(\nabla \log g_{\theta,\boldsymbol{\omega};M+1}(X))^{T}]
.\] Easy but tedious computations give
$$
\E^{\star}[(\nabla \log g_{\theta,\boldsymbol{\omega};M}(X))(D(X))^{T}]=\left(\begin{array}{ccc} 0 &\cdots &0\\
\vdots & \vdots & \vdots \\0 &\cdots &0
\end{array}\right),
$$
so that
$$
K^{(M+1)} =\left(\begin{array}{cc} J_{M} & 0\\
0&0
\end{array} \right) + \Delta
$$
where   $\Delta=\E^{\star}[D(X)D(X)^{T}]$  is positive semi-definite.
As said before, $\tilde{J}_{M+1}$ is obtained from $K^{(M+1)}$ using the similar formula as from $J_{M+1}$. 
Then usual algebra gives  that $\tilde{J}_{M+1}\geq \tilde{J}_{M}$ since  $\Delta$ is positive semi-definite. 

\subsection{Proof of Lemma \ref{lem:convscore}}

Under (A1), 
the functions $f_{j,c}^\star$ are upper bounded. 
 Let, for any $M$, $\mathbb{A}_{M}$ be  the orthogonal projection  in  $L^{2}(g_{\theta^{\star},\mathbf{f}^\star}d\mathbf{x})$ onto ${\dot{\cal P}}_{M}$, the set of step functions spanned by the functions $\left(S^{\star}_{\boldsymbol{\omega},M}\right)_{j,c,m}$,  $j=1,\ldots,k$, $c=1,2,3$, $m=1,\ldots,M-1$. Then for all $j=1,\ldots,k-1$, 
$$
(\tilde{\psi}_{M})_{j}=\left(S^{\star}_{\theta,M}\right)_{j}-{\mathbb{A}}_{M} \left(S^{\star}_{\theta,M}\right)_{j},
$$
so that
\begin{equation}\label{eq:psi_tilde_decomp_project}
(\tilde{\psi})_{j}-(\tilde{\psi}_{M})_{j}=\left(S^{\star}_{\theta}\right)_{j}-\left(S^{\star}_{\theta,M}\right)_{j} - {\mathbb{A}}_{M}\left[  \left(S^{\star}_{\theta}\right)_{j}-\left(S^{\star}_{\theta,M}\right)_{j}\right]
+\left({\mathbb{A}}_{M}- {\mathbb{A}}\right)  \left(S^{\star}_{\theta}\right)_{j}.
\end{equation}
\noindent Using Assumption (A2) and Corollary 1.7 in Chapter 3 of 
\cite{MR2129625}, we have that as $M$ tends to infinity, $\left(S^{\star}_{\theta,M}\right)_{j}$ converges to $\left(S^{\star}_{\theta}\right)_{j}$ Lebesgue a.e. Both functions are uniformly upper bounded by the finite constant $1/\theta_{j}^{\star}$ using Assumption (A.1), so that $\left(S^{\star}_{\theta,M}\right)_{j}$ converges to $\left(S^{\star}_{\theta}\right)_{j}$ in $L^{2}(g_{\theta^{\star},\mathbf{f}^{\star}}(\mathbf{x})d\mathbf{x})$ as $M$ tends to $+\infty$ and $\left\|\left(S^{\star}_{\theta}\right)_{j}-\left(S^{\star}_{\theta,M}\right)_{j}
\right\|_{L^{2}(g_{\theta^{\star},\mathbf{f}^{\star}}d\mathbf{x})}$
converges to $0$  as $M$ tends to $+\infty$. 
Thus to prove that $(\tilde{\psi})_{j}-(\tilde{\psi}_{M})_{j}$ converges to $0$   in $L^{2}(g_{\theta^{\star},\mathbf{f}^\star}d\mathbf{x})$ when $M$ tends to $+\infty$,
we need only to prove that $\|\left({\mathbb{A}}_{M}- {\mathbb{A}}\right)  \left(S^{\star}_{\theta}\right)_{j}\|_{ L^2(g_{\theta^{\star},\mathbf{f}^{\star}})}$ converges to $0$. 
So we now prove that, for all $S \in L^2(g_{\theta^{\star},\mathbf{f}^{\star}})$, $\|{\mathbb{A}}_M S -{\mathbb{A}}S\|_{ L^2(g_{\theta^{\star},\mathbf{f}^{\star}})}$ converges to $0$ when $M$ tends to $+\infty$. \\
First we prove that ${\mathbb{A}}_M S$ converges in $ L^2(g_{\theta^{\star},\mathbf{f}^{\star}})$. Let 
 $L \geq M$ and set $\psi_S(M,L)  =\| {\mathbb{A}}_M S - {\mathbb{A}}_L S\|_{L^2(g_{\theta^{\star},\mathbf{f}^{\star}})}$. For large enough $M$, we have
 $\psi_S(M,L) = \| {\mathbb{A}}_M({\mathbb{A}}_LS) -{\mathbb{A}}_LS\|_{L^2(g_{\theta^{\star},\mathbf{f}^{\star}})} $, 
  since using (A3), ${\dot{\cal P}}_{M}  \subset  {\dot{\cal P}}_L$. It is easy to see that, for all $M$, $\psi_S(M,L)$ is  a monotone sequence, non decreasing  in 
  $L$, and bounded,  so that it converges to some $\psi^\star_S(M) \geq 0$. Moreover, since for all $L$,
  $\psi_S(M+1,L) \leq \psi_S(M,L)$, at the limit $\psi^\star_S(M)$ is monotone non increasing in $M$ and non-negative so that it converges. Let $\psi_S^\star$ be its limit. 
Because
\begin{align*} &{\mathbb{A}}_L S = {\mathbb{A}}_M ({\mathbb{A}}_L S) + (I-{\mathbb{A}}_M)({\mathbb{A}}_L S) = {\mathbb{A}}_M S + (I-{\mathbb{A}}_M)({\mathbb{A}}_L S), \\ 
&{\mathbb{A}}_M S \perp (I-{\mathbb{A}}_M)({\mathbb{A}}_L S),
\end{align*}
we get that
  $$\|{\mathbb{A}}_L S\|^2 = \|{\mathbb{A}}_M S\|^2 +  \|(I-{\mathbb{A}}_M)({\mathbb{A}}_L S)\|^2 = \|{\mathbb{A}}_M S\|^2 +  \psi_S(M,L)^2.$$
  Let $M$ be fixed. Then we get that 
  $$ \lim_{L\rightarrow +\infty }\| {\mathbb{A}}_L S\|^2 = \|{\mathbb{A}}_M S\|^2+ \psi^\star_S(M)^2$$
 and if we write $\ell$ the limit on the lefthandside of the equation, by letting now $M$ tend to infinity we get that
 $\ell = \ell + \psi_S^\star $. This in turns implies that  $\psi_S^\star = 0$. Now let $L_p, M_p$ converge to infinity as $p$ goes to infinity
 in such a way that for all $p$
 $L_p \geq M_p$ (which we can always assume by symmetry). Then
$$ \psi_S(M_p, L_p) \leq \psi_S^\star(M_p) \stackrel{p\rightarrow +\infty }{\rightarrow } 0
$$ so that the sequence ${\mathbb{A}}_MS$ is Cauchy in $L^2(g_{\theta^{\star},\mathbf{f}^{\star}})$ and converges. 
{Denote $\overline{\mathbb{A}} S$ its limit. Let us prove that $\overline{\mathbb{A}} S \in  {\dot{\cal P}}$. 
Any function in ${\dot{\cal P}}_M$ is a finite linear combination of functions $S_{j,c,M}$ of form
\begin{equation}
\label{help}
S_{j,c,M}(\mathbf{x})
=\frac{h_{j,c,M}(x_{c})\prod_{c'}f_{\omega^{\star}_{j,c';M}}(x_{c'})}{g_{\theta^{\star},\boldsymbol{\omega}^{\star}_{M};M}(\mathbf{x})}
\end{equation}
with $h_{j,c,M}\in {\cal H}_{j,c}$ is such that $h_{j,c,M} f_{\omega^{\star}_{j,c;M}}$ is a linear combination of indicator functions. It is thus enough to prove that the limit of any converging sequence of such functions is in ${\dot{\cal P}}$. 
Consider a sequence $S_{j,c,M}$ converging  to some variables $S^\star$ in $L^2(g_{\theta^{\star},\mathbf{f}^{\star}})$ as $M$ tends to infinity. Note that
$$ h_{j,c,M}(x_{c}) = \sum_{\ell = 1 }^k \theta^\star_\ell S_{j,c,M}(\mathbf{x}) \times \prod_{c'} \frac{ f_{\omega^{\star}_{\ell,c';M}}(x_{c'}) }{  f_{\omega^{\star}_{j,c';M}}(x_{c'})
} $$
so that condition \eqref{moment:cond} together with the almost sure convergence of $ f_{\omega^{\star}_{\ell,c';M}}(x_{c'})/ f_{\omega^{\star}_{j,c';M}}(x_{c'})$  towards $ f^{\star}_{\ell,c'}(x_{c'}) / f^{\star}_{j,c'}(x_{c'})$ implies that 
 $$ S_{j,c,M}(\mathbf{x}) \times \prod_{c'} \frac{ f_{\omega^{\star}_{\ell,c';M}}(x_{c'}) }{  f_{\omega^{\star}_{j,c';M}}(x_{c'})}\stackrel{L^2(g_{\theta^{\star},\mathbf{f}^{\star}})}{\rightarrow } S^\star\prod_{c'}\frac{ f^{\star}_{\ell,c'}(x_{c'}) }{ f^\star_{j,c'}(x_{c'})} := h,$$
 which in turns implies that $h_{j,c,M}$ converges in $L^2(f^{\star}_{j,c})$ to $h$.  
 Since $h_{j,c,M} \in \mathcal H_{j,c} $ for all $M$, $h \in \mathcal H_{j,c} $ and $S^\star \in  {\dot{\cal P}}$.
\\
}
We now prove that all function in $\dot{\cal P}$ is a limit in $L^2(g_{\theta^{\star},\mathbf{f}^{\star}})$ of functions in ${\dot{\cal P}}_M$. As before, it is enough to prove it for functions $S_{j,c}$ of form
$$
S_{j,c}(\mathbf{x})
=\frac{h(x_{c})\prod_{c'}f^{\star}_{j,c'}(x_{c'})}{g_{\theta^{\star},\mathbf{f}^{\star}}(\mathbf{x})}
$$
with $h\in {\cal H}_{j,c}$. We are thus looking for a sequence of functions $h_{M}\in  {\cal H}_{j,c}$  such that $h_{M} f_{\omega^{\star}_{j,c;M}}$ is a linear combination of indicator functions and such that $S_{j,c,M}$ as defined by (\ref{help}) converges to $S_{j,c}$ in $L^2(g_{\theta^{\star},\mathbf{f}^{\star}})$. Using Lemma 1.2 in \cite{MR2129625}, $h$ may be approximated by a continuous function, which in turns may be approximated by a (centred) linear combination of indicator functions and the result follows using again that  $\prod_{c'}f_{\omega^{\star}_{j,c';M}}/{g_{\theta^{\star},\boldsymbol{\omega}^{\star}_{M};M}}$ and $\prod_{c'}f^{\star}_{j,c'}/{g_{\theta^{\star},\mathbf{f}^{\star}}}$ are bounded (by $1/ \min_{j} \theta_{j}^{\star}$).
{Thus, we easily get that for all $S$ in $L^2(g_{\theta^{\star},\mathbf{f}^{\star}})$, ${\mathbb{A}}_M {\mathbb{A}} S- {\mathbb{A}}S$
converges to $0$ in $L^2(g_{\theta^{\star},\mathbf{f}^{\star}})$, so that $\overline{\mathbb{A}}{\mathbb{A}}={\mathbb{A}}$.\\
 Now, one easily deduces  that $\overline{\mathbb{A}}={\mathbb{A}}$. Indeed: if $S$ is in $\dot{\cal P}$, one has ${\mathbb{A}}S=S$, and then $\overline{\mathbb{A}}S=S$. If now $S$ is in the orthogonal of $\dot{\cal P}$, then for any $\tilde{S}\in \dot{\cal P}$, one has $\langle {\mathbb{A}}_{M} S, \tilde{S} \rangle =\langle  S, {\mathbb{A}}_{M}\tilde{S} \rangle$ which leads to $\langle \overline{\mathbb{A}} S, \tilde{S} \rangle = \langle S, \overline{\mathbb{A}}\tilde{S}\rangle =0$, so that $\overline{\mathbb{A}} S$ is in the orthogonal of $\dot{\cal P}$ and in $\dot{\cal P}$ so that 
$\overline{\mathbb{A}} S$=0.}

\subsection{Proof of Proposition \ref{prop:limit_thetaM}}
\label{sec:proofProp4}
Proposition \ref{prop:limit_thetaM} is easily implied by Lemma \ref{le:limit_thetaM} which formalizes the following.
When the sequence of observations $X_1, \dots, X_n$ and $n$ are fixed, then almost surely there exists a sufficiently fine partition ${\cal I}_M$ such that there exists at most one component of an observation in each set $I_m$, $m\leq M$. Then we can reorder the sets $I_m$ so that $X_{i,c} \in I_{i+n(c-1)}$, for all $c\in \{1,2,3\}$ and $i\leq n$. In this case, the likelihood $\ell_n(\cdot,\cdot;M)$ is maximised at each parameter $(\theta, \boldsymbol{\omega})$ belonging to the set $\mathcal{S}_M \subset \Delta_k \times (\Delta_M)^{3k} $ that we explain now (and formalise in  Lemma \ref{le:limit_thetaM}). Each element of $\mathcal{S}_M$ corresponds to one clustering of the observations in $k$ sets (represented by the $(A_j^\star)_{j\leq k}$ in Lemma \ref{le:limit_thetaM}) of size as equal as possible. For each clustering, for all $j\leq k$,\\
$\bullet$ $\theta_j=\#A^\star_j/n$ is the proportion of observations associated to $A^\star_j$
 (then the $\theta_j$ are almost equal to $1/k$),\\
$\bullet$ for all $c\in \{1,2,3\}$ and for all $l\leq M$,
\\
\begin{equation*}
\omega_{j,c,l}=
\left\{\begin{array}{ll}
1/\#A^\star_j & \text{ if } l-n(c-1)\in A^\star_j \text{ (i.e. $X_{l-n(c-1)}\in I_l$ is associated to } \\
& \text{the hidden state $j$),}\\
0 & \text{ if } l-n(c-1)\in \{1, \dots n\} \setminus A^\star_j
\text{ (i.e. $X_{l-n(c-1)}\in I_l$ is not}\\
& \text{ associated to $j$),}\\
0 & \text{ otherwise (i.e. there is no observation in $I_l$).}
\end{array}
\right.
\end{equation*} 

  \begin{lemma}\label{le:limit_thetaM}
   Let $X_1, \dots, X_n$ be fixed observations,
 as soon as for all $i\leq n$ and $c\in \{1,2,3\}$,  $X_{i,c}\in I_{i+n(c-1)}$ then the likelihood $\ell_n(\cdot,\cdot;M)$ is maximised at  $(\widehat{\theta}_M,\widehat{\boldsymbol{\omega}}_M)$ if and only if  $(\widehat{\theta}_M,\widehat{\boldsymbol{\omega}}_M)\in \mathcal{S}_M$ where
 \begin{align*}
 \mathcal{S}_M
 =\big\{ (\theta,\boldsymbol{\omega}): ~&
\theta_j= \#A^\star_j/n, ~\omega_{j,c,l}=\mathds{1}_{l-n(c-1) \in A^\star_j}/\#A^\star_j,\\
 & (J_1,J_2) \text{ partition of } \{1, \dots,k\}, ~\#J_2= n- k \lfloor n/k \rfloor =: r\\
 & (A^\star_j)_{j\leq k}
 \text{ partition of  }
 \{1, \dots, n\},
  \#A^\star_{j_1}=\lfloor n/k \rfloor=:q, \text{ for }  ~ j_1 \in J_1,\\& ~ \#A^\star_{j_2}=\lfloor n/k \rfloor +1=: q+1, \text{ for } j_2 \in J_2 \big\},
 \end{align*}
 and $n=kq+r, ~ 0\leq r \leq k-1$.
  \end{lemma}

\begin{proof}

Since the set of parameters is compact and the likelihood is a continuous function of the parameters then the maximum is attained. 

 If $(\theta,\boldsymbol{\omega})$ maximises the likelihood $\ell_n(\cdot,\cdot;M)$, 
 \begin{enumerate}[label=(P\arabic*)]
 \item\label{hyp:P1} then, for all $1\leq i \leq n$, there exists $1 \leq j \leq k$ such that $\omega_{j,c,i+n(c-1)}>0$ for all $c \in \{1,2,3\}$. 
 \\
 Indeed, if there exists $1\leq i \leq n$ such that for all $1 \leq j \leq k$, $\omega_{j,c,i+n(c-1)}=0$ for some $c\in\{1,2,3\}$, then   
\begin{multline*}
\ell_n(\theta,\omega;M)
= \sum_{i=1}^n \log\left( \sum_{j=1}^k \theta_j \prod_{c=1}^3 \omega_{j,c,i+n(c-1)}  \right)
\\+ \underbrace{ \sum_{i=1}^n \log\left( 1/(|I_i| |I_{i+n}| |I_{i+2n}| ) \right)}_{\text{constant}}=-\infty
.\end{multline*} 
 \item\label{hyp:P2} and if there exists $j,c,i$ such that $\omega_{j,c,i+n(c-1)}=0$ and $\theta_j>0$  then $\omega_{j,d,i+n(d-1)}=0$ for all $d$. 
 \\
 Indeed otherwise you can give the weight $\omega_{j,d,i+n(d-1)}$, to one of the other $\omega_{j,d,s+n(d-1)}$ for which $\omega_{j,e,s+n(e-1)}>0$, for all $e\neq d$ (which exist otherwise take $\theta_j=0$ which would increase the likelihood) and this increases the likelihood.
 \item\label{hyp:P3} and if  $\theta_j>0$, then $\omega_{j,c,l}=0$ if $l-n(c-1)\notin \{1, \dots,n\}$. \\
 Indeed, in this case, there is no observation in $I_l$ so that $\omega_{j,c,l}$ does not appear in the likelihood and we conclude similarly as the previous point.
\end{enumerate}   
 Combining all the previous remarks, we know that the maximum can only be attained (and is at least once) in one of the following sets, indexed by $J \subset \{1, \dots,k\}$ which determines the zeros of $\theta$  and  $A_j \subset\{1, \dots, n\}$, $j\leq k$, which determine the zeros of $\omega$: 
\begin{align*} 
\mathcal{S}_{J,A_1, \dots, A_k} 
= & \{ \theta \in \Delta_k: ~ \theta_j>0, ~  j \in J, \theta_j=0, ~ j \in J^c \} \\
	& \times \prod_{j \leq k } \Big\{ (\omega_{j,1,\cdot},\omega_{j,2,\cdot},\omega_{j,3,\cdot}) \in (\Delta_M)^3: \\
	& \hspace*{1.5cm} \text{ if }j\in J, ~ \underset{
	\text{\rotatebox{90}{$\Rsh$} using \ref{hyp:P2}
	\qquad \raisebox{-.5mm}{\rotatebox{90}{$\Lsh$}} using \ref{hyp:P3}}}{\omega_{j,c,i + n(c-1)}>0}, \text{ if } ~ i \in A_j, ~ c \in\{1,2,3\}\\
	& \hspace*{1.5cm} \text{ and } \omega_{j,c,l}=0, \text{ if } l \in \{1,\dots,M\} \setminus \{ i + n(c-1), ~ i\in A_j\}  \Big\}
.
\end{align*}
Note that we do not assume that  $(A_j)_{j\in J}$ is  a partition of $\{1, \dots, n\}$.

We fix $J \subset \{1, \dots,k\}$ and $A_j \subset\{1, \dots, n\}$, $j\in J$. Now we search for parameters $(\bar{\theta}, \bar{\boldsymbol{\omega}})$ in $\mathcal{S}_{J,A_1, \dots, A_k}$  which maximize the likelihood. They are zeros of the derivative of 
\begin{equation}\label{eq:deriv_log_vraiss}
(\theta,\boldsymbol{\omega},\lambda, \mu)\mapsto \ell_n(\theta,\boldsymbol{\omega};M) + \lambda\left(\sum_{j=1}^k \theta_j-1\right) + \sum_{c=1}^3 \mu_{j,c} \left(\sum_i \omega_{j,c,i} -1\right),
\end{equation}
with respect to non zero components ($\theta_j$, $\omega_{j,c,i+n(c-1)}$, $\lambda$ and $\mu_{j,c}$, for $j \in J$, $i \in A_j$, $1\leq c \leq 3$).
Annulling the partial derivatives give
\begin{align}
\sum_{i \in A_j} \frac{\bar{\omega}_{j,1,i}\bar{\omega}_{j,2,i+n}\bar{\omega}_{j,3,i+2n}}{\sum_{s \in J(i)} \bar{\theta}_s \bar{\omega}_{s,1,i}\bar{\omega}_{s,2,i+n}\bar{\omega}_{s,3,i+2n}}
 & =  - \lambda, & \forall j \in J
\label{eq:diffe_ln_pj} \\
\frac{\bar{\theta}_j \prod_{d\neq c} \bar{\omega}_{j,d,i+n(d-1)}}{\sum_{s \in J(i)} \bar{\theta}_s \bar{\omega}_{s,1,i}\bar{\omega}_{s,2,i+n}\bar{\omega}_{s,3,i+2n}}
& = - \mu_{j,c} ,& \forall j\in J, ~ i \in A_j, ~ c \in \{1,2,3\}
\label{eq:diffe_ln_omega}\\
\sum_{j\in J} \bar{\theta}_j & = 1 ,& 
\label{eq:diffe_ln_lambda}\\
\sum_{i\in A_j} \bar{\omega}_{j,c,i+n(c-1)} &= 1 ,&\forall j\in J, ~ c \in \{1,2,3\},
\label{eq:diffe_ln_mu}
\end{align}
where $J(i)=\{ s \in J: ~ i\in A_s\}$.

Multiplying Equation \eqref{eq:diffe_ln_omega} by  $\bar{\omega}_{j,c,i+n(c-1)}$ and then summing the result over $i\in A_j$ and using Equation \eqref{eq:diffe_ln_mu}, we obtain that $\mu_{j,c}$ does not depend on $c$. Then using Equations \eqref{eq:diffe_ln_omega} for $c=1$, $c=2$ and $c=3$, we obtain
\[\bar{\theta}_j  \bar{\omega}_{j,1,i}\bar{\omega}_{j,2,i+n} = 
\bar{\theta}_j  \bar{\omega}_{j,1,i}\bar{\omega}_{j,3,i+2n}=
\bar{\theta}_j  \bar{\omega}_{j,2,i+n}\bar{\omega}_{j,3,i+2n}
,\]
so that 
\begin{equation}\label{eq:egalite_omega_c}
\bar{\omega}_{j,1,i}=\bar{\omega}_{j,2,i+n}=\bar{\omega}_{j,3,i+2n}
.
\end{equation}

Furthermore, multiplying Equation \eqref{eq:diffe_ln_pj} by $\bar{\theta}_j $  and summing the result over $j \in J$ and using Equation \eqref{eq:diffe_ln_lambda}, we obtain $\lambda=-n$.
 Moreover by multiplying Equation \eqref{eq:diffe_ln_omega} by $\bar{\omega}_{j,c,i+n(c-1)}$, and then summing the result over $i \in A_j$ and finally subtracting \eqref{eq:diffe_ln_pj} multiplied by $\bar{\theta}_j$ to the result (ie making $\sum_{i\in A_j} (-\bar{\theta}_j) \eqref{eq:diffe_ln_pj} + \bar{\omega}_{j,c,i+n(c-1)} \eqref{eq:diffe_ln_omega}  $), we get 
 \begin{equation}\label{eq:egalite_muc}
0= - \mu_{j,c} -n\bar{\theta}_j.
 \end{equation}
  Then using again Equations \eqref{eq:diffe_ln_omega}, \eqref{eq:egalite_omega_c} and \eqref{eq:egalite_muc}, we get
\begin{equation*}
\bar{\omega}^2_{j,c,i+n(c-1)}=n \sum_{s\in J(i)} \bar{\theta}_s \bar{\omega}^3_{s,1,i}, \quad \forall j \in J(i), ~ \forall c  \in \{1,2,3\},
\end{equation*}
 so that $\bar{\omega}_{j,c,i+n(c-1)}$ does not depend on $j \in J(i)$ and 
 \begin{equation}\label{eq:zeros_omega}
 \bar{\omega}_{j,c,i+n(c-1)}=\mathds{1}_{i\in A_j}/\left(n \sum_{s \in J(i)} \bar{\theta}_s \right), \quad \forall j \in J(i).
 \end{equation}

 For each $\mathcal{S}_{J,A_1, \dots, A_k}=:\mathcal{S}$, we have obtained the zeros of the derivative of the log-likelihood, that we now denote $(^\mathcal{S}\bar{\theta},^\mathcal{S}\!\bar{\boldsymbol{\omega}})$, to emphasize the dependence with the considered set $\mathcal{S}$. We now want to know which of these zeros $(^\mathcal{S}\bar{\theta},^\mathcal{S}\!\bar{\boldsymbol{\omega}})$ are local maxima thanks to the second partial derivatives. 
 
We consider sets $\mathcal{S}_{J,A_1, \dots, A_k}$ for which there exists $i\leq n$ such that there exist $j$ and $l$ are in $J(i)$ and $j \neq i$. We consider a second partial derivative of 
\[
\tilde{\ell}_n(\theta, \tilde{\boldsymbol{\omega}};M)
=
\sum_{i=1}^n \log\left(\sum_{j=1}^k \theta_j (\tilde{\omega}_{j,1,i})^3 \right)
\]
 that is the log-likelihood (up to an additive constant) associated to the model where for all $1\leq m \leq k$, $1 \leq s \leq n$, 
 $\omega_{m,1,s}=\omega_{m,2,s+n}=\omega_{m,3,s+2n}$.
Assume without loss of generality that $\theta_l\geq \theta_j$, then (using that $\theta_k=1-\sum_{m<k}\theta_m$ and $\omega_{j,1,n}=1-\sum_{s<n}\omega_{j,1,s}$),
\[
\frac{\partial^2 \tilde{\ell_n}}{ \partial \tilde{\omega}_{j,1,i}^2}(^\mathcal{S}\bar{\theta},^\mathcal{S}\!\bar{\boldsymbol{\omega}};M)
= C \left(6 ~ ^\mathcal{S}\bar{\theta}_j \!\!\!\!\!\!\!\!\sum_{m\in J(i) \setminus \{ j\}} \!\!\!\!\!\!\!\!\!\! ~^\mathcal{S}\bar{\theta}_m -3 ~ ^\mathcal{S}\bar{\theta}_j^2\right)
\geq  C \left(6 ~ ^\mathcal{S}\bar{\theta}_j  ~ ^\mathcal{S}\bar{\theta}_l -3  ~ ^\mathcal{S}\bar{\theta}_j^2\right)
>0
,\] 
 where $C>0$. This implies that for all sets $\mathcal{S}_{J,A_1, \dots, A_k}:={\cal S}$ where there exists $i\leq n$ such that $\#J(i)>1$, every zeros $(^\mathcal{S}\bar{\theta},^\mathcal{S}\!\bar{\boldsymbol{\omega}})$ is not a local maximum.
 So that the only possible local maxima of $\ell_n(\theta,\boldsymbol{\omega};M)$ are the zeros $(^{\mathcal{S}_{J,A_1, \dots, A_k}}\bar{\theta},^{\mathcal{S}_{J,A_1, \dots, A_k}}\!\bar{\omega})$ where $\#J(i)=1$ for all $i\leq n$, i.e. when $(A_j)_{j\in J}$ forms a partition of $\{1, \dots, n\}$. 
 
 So we now only consider sets $A_j$, ${j\in J}$  which form a partition of $\{1, \dots, n\}$ and  $\bar{\omega}_{j,c,i+n(c-1)}=\mathds{1}_{i\in A_j}/ (n \bar{\theta}_j)$ for $i\in A_j$, using Equation \eqref{eq:zeros_omega}. As $\sum_{i\in A_j}\! \bar{\omega}_{j,1,i}\! = 1$, we then obtain that 
 $\bar{\theta}_j=  \#A_j/n=1/ (n \bar{\omega}_{j,1,i}) $, for all $i\in A_j$. So that we now only have to choose the best partition
  $(A_j)_{j \in J}$ of $\{1, \dots, n\}$ and $J$.
   Let $N_j=\#A_j$, we know that $\sum_j N_j=n$ and the log-likelihood at the local maximum   $(^\mathcal{S}\bar{\theta}, ^\mathcal{S}\!\bar{\omega})$ associated to $\mathcal{S}_{J,A_1, \dots, A_k}=:\mathcal{S}$ is 
 \[
 \ell_n(^\mathcal{S}\bar{\theta}, ^\mathcal{S}\!\bar{\omega};M)=\sum_{s \in J} N_s \log (N_s^{-2})+constant.
 \]
 So that we want to minimize 
 \begin{equation}\label{eq:minimi_avec_J}
 \sum_{s\in J} N_s \log (N_s) \text{ under the constraint }\sum_{s\in J} N_s=n
 \end{equation}
  over $J \subset \{1,\dots k\}$ and $N_j\in \mathbb{N}$, $j\in J$. This minimization is equivalent to the minimization of 
 \begin{equation}\label{eq:minimi_sans_J}
\sum_{s\leq k} N_s \log (N_s) \text{ under the constraint }\sum_{s\leq k} N_s=n
 \end{equation}
  over $N_j\in \mathbb{N}$, $j\leq k$ (since then the problem \eqref{eq:minimi_sans_J} is less constrained than for the minimization of \eqref{eq:minimi_avec_J} when $J$ is fixed).
 
 And, when $k$ divides $n$, the minimum of \eqref{eq:minimi_sans_J} is attained at $N_s=n/k$. Otherwise, when $k$ does not divide $n$, consider only two indices  $s_1$, $s_2$ in $\{1, \dots, k\}$ and assume that $N_s$, $s\notin\{s_1,s_2\}$ are fixed such that $N_{s_1}+N_{s_2}=S_N$ is also fixed. Then we want to minimise $-N_{s_1}\log (N_{s_1}) - (S_N-N_{s_1}) \log (S_N-N_{s_1})$.
 Studying the function $x \in(0,S_N) \mapsto -x \log(S_N) -(S_N-x) \log (S_N-x)$, we obtain that the minimum is attained when $N_{s_1}$ and $N_{s_2}=S_N-N_{s_1}$ are the closest of $N_S/2$. Then in both cases, the MLE is attained at 
 every $(\theta,\boldsymbol{\omega}) \in \mathcal{S}_M$.

\end{proof}

\subsection{Proof of Corollary \ref{co:bound_for M_n}}
\label{sec:proof_cor_lim_Mn}
Suppose that for all $N>0$ and all $C>0$, there exists $n\geq N$ such that 
\[
n^2 \left(\max_{m\leq M_n} |I_m|\right)^2 M_n \leq C.
\]
So that there exists a subsequence $(\phi(n))_{n\in \mathbb{N}}$ of $(n)_{n\in \mathbb{N}}$ such that
\begin{equation}\label{eq:sous_suite_conv_0}
(\phi(n))^2 \left(\max_{m\leq M_{\phi(n)}} |I_m| \right)^2 M_{\phi(n)} \underset{n \to \infty}{\longrightarrow} 0
.
\end{equation}
Set $\epsilon >0$, by Proposition \ref{prop:limit_thetaM}, there exists $N_1>0$ such that for all $n\geq N_1$,
\begin{align}
P&\left( \left| \widehat{\theta}_{M_n}(X_{1:\phi(n)}) -(1/k, \dots, 1/k) \right| \leq \epsilon \right) 
\nonumber\\
& \geq P\left( \left\{  \exists ~ 1\leq i_1,i_2 \leq \phi(n), ~ 1 \leq c,d \leq 3, ~ m \leq M_{\phi(n)} :~  X_{i_1,c}\in I_m, X_{i_2,d} \in I_m  \right\}^c \right)
\nonumber\\
& \geq 1- \sum_{ i_1 =1}^{\phi(n)} \sum_{ i_2 =1}^{\phi(n)} \sum_{ m=1 }^{ M_{\phi(n)}} P\left(X_{i_1,c}\in I_m, X_{i_2,d} \in I_m \right)
\nonumber\\
& \geq 1- (\phi(n))^2 M_{\phi(n)} \max\left( \sup g, (\sup g)^2 \right)   \left(\max_{m \leq M_{\phi(n)}} |I_m| \right)^2 
. \label{eq:lim_contradiction}
\end{align}
Using Equations \eqref{eq:sous_suite_conv_0} and \eqref{eq:lim_contradiction} and Assumption (A3), then $\widehat{\theta}_{M_n}(X_{1:\phi(n)})$ tends in probability to $(1/k, \dots, 1/k)$ which contradicts the convergence in law of $\widehat{\theta}_{M_n}$ to $\theta^\star$. This concludes the proof.

\subsection{Proof of Theorem \ref{th:construction_Mn}}
\label{sec:proof_th_Mn}
We first recall Lemma 2.1 in  \cite{Arl:2014:hdr}:

\begin{lemma}[Sylvain Arlot]\label{le:Arlot}
Let $A,B,C,R: {\cal M} \to \mathbb{R}$.
If for all $m,m' \in {\cal M}$,
\[
(C(m)-R(m))-(C(m')-R(m')) \leq A(m) +B(m'),
\]
then for all $\widehat{m} \in {\cal M}$ such that 
$C(\widehat{m}) \leq \inf_{m \in {\cal{M}}} C(m) +\rho $, $\rho>0$,
\[
R(\widehat{m})-B(\widehat{m}) \leq \inf_{m \in {\cal{M}}} 
\{R(m)+A(m)\} +\rho  
.\]
\end{lemma}

We are going to use this lemma with $R({\cal I})=R_{a_n}({\cal I})$, $C({\cal I})=C_{CV}({\cal I})$ and
\[A({\cal I})=B({\cal I})= \epsilon_n R({\cal I}) + \delta_n.
\]
Using Hoeffding's inequality,
\begin{align*}
P\left(\{ - B({\cal I}) \leq C_{CV}({\cal I}) - R_{a_n}({\cal I}) 
{- \E^{\star}\left[ \|\tilde{\theta}_{{\cal I}_0} (X_{B_{-b}}) - \theta^\star \|^{2}_{\mathcal{T}_k} \right]}
 \leq A({\cal I}) \}^c\right)
 \\ \leq  2 \exp\left( -2 b_n A({\cal I})^2 \right),
\end{align*}
since $\|\tilde{\theta}_{{\cal I}}(X_{B_b})-\tilde{\theta}_{{\cal I}_0} (X_{B_{-b}})\|^{2} \leq 1$, for all $b$.
We introduce the sets
\begin{equation}\label{eq:hoeffding}
{\cal S}_{\cal I} = \left\{ - B({\cal I}) \leq C_{CV}({\cal I}) - R_{a_n}({\cal I}) 
{-  
 \E^{\star}\left[ \|\tilde{\theta}_{{\cal I}_0} (X_{B_{-b}}) - \theta^\star \|^{2}_{\mathcal{T}_k} \right]}
\leq A({\cal I}) \right\}
\end{equation}
for all ${\cal I} \in {\cal M}_n$.
Using Lemma \ref{le:Arlot},
on the set $\cap_{{\cal I} \in {\cal M}_n} {\cal S}_{\cal I}$,
Equation \eqref{eq:ineq_selection_model} holds and using Equation \eqref{eq:hoeffding}, we obtain
\[
P(\cap_{{\cal I} \in {\cal M}_n} {\cal S}_{\cal I})
\geq
  1-2m_n \exp{\left(-2b_n \left(\epsilon_n \inf_{ {\cal I} \in {\cal M}_n } R_{a_n}({\cal I}) + \delta_n\right)^2\right)}
  .\]

\subsection{Proof of Proposition \ref{prop:alm_oracle_ineq}}
\label{sec:proof_le_Mn}
Using Theorem \ref{th:construction_Mn},
\begin{align*}
&\E^{\star} \left[ a_n R_{a_n}(\widehat{{\cal I}}_n)\right] \\
&\leq
 a_n \left(\frac{1+ \epsilon_n}{1- \epsilon_n} \inf_{ {\cal I} \in {\cal M}_n } R_{a_n}({\cal I}) + \frac{ 2 \delta_n }{ 1 -\epsilon_n } \right)\\
&\quad +
2 a_n m_n
\exp \left( -2b_n \left(\epsilon_n \inf_{ {\cal I} \in {\cal M}_n } R_{a_n}({\cal I}) + \delta_n\right)^2 \right)
\end{align*}
we can conclude by taking
$\epsilon_n=\delta_n = 1/(\log(n) a_n) $.

\subsection*{Acknowledgements}
This work was partly supported by the grants ANR Banhdits and Calibration. 
{ We want to thank the reviewers and the associate editor for their helpful comments.}

\bibliographystyle{plain}
\bibliography{biblio}
 \end{document}